\documentclass[a4paper,final]{amsart}

\usepackage{amsmath}%
\usepackage{amsfonts}%
\usepackage{amssymb}%
\usepackage{amsthm}
\usepackage{graphicx}
\usepackage{dsfont}
\usepackage{pdfcomment}
\usepackage{tikz}
\usepackage{tikz-3dplot}
\usepackage{subfigure}

\usepackage{thmtools}
\usepackage{thm-restate}
\usepackage{hyperref}
\usepackage{cleveref}
\usepackage{esvect}

\usepackage{marvosym}

\theoremstyle{plain}
\newtheorem{theorem}{Theorem}[section]
\newtheorem{lemma}[theorem]{Lemma}
\newtheorem{corollary}[theorem]{Corollary}
\theoremstyle{definition}
\newtheorem{remark}[theorem]{Remark}
\newtheorem{definition}[theorem]{Definition}
\newtheorem{example}[theorem]{Example}

\newcommand*{\LargerCdot}{\raisebox{-0.25ex}{\scalebox{1.2}{$\cdot$}}}
\newcommand{\field}[1]{\mathbb{#1}}
\DeclareMathOperator{\rank}{rank}
\DeclareMathOperator{\A}{A}
\DeclareMathOperator{\divergence}{div}
\DeclareMathOperator{\curl}{curl}
\DeclareMathOperator{\Ker}{Ker}
\DeclareMathOperator{\Real}{Re}
\DeclareMathOperator{\Imaginary}{Im}
\DeclareMathOperator{\grad}{grad}
\DeclareMathOperator{\lcrs}{lcr}
\renewcommand{\Re}{\Real}
\renewcommand{\div}{\divergence}
\renewcommand{\Im}{\Imaginary}

\title[Isothermic triangulated surfaces]{Isothermic triangulated surfaces}
\author{Wai Yeung Lam}
\author{Ulrich Pinkall}

\address{Wai Yeung Lam\\
	Technische Universit\"at Berlin\\Institut f\"ur Mathematik\\
	Stra{\ss}e des 17.\ Juni 136\\
	10623 Berlin\\ Germany}

\address{Ulrich Pinkall\\
	Technische Universit\"at Berlin\\Institut f\"ur Mathematik\\
	Stra{\ss}e des 17.\ Juni 136\\
	10623 Berlin\\ Germany}

\email{lam@math.tu-berlin.de, pinkall@math.tu-berlin.de}

\begin{document}
	
	\begin{abstract}
		We found a class of triangulated surfaces in Euclidean space which have similar properties as isothermic surfaces in Differential Geometry. We call a surface isothermic if it admits an infinitesimal isometric deformation preserving the mean curvature integrand locally. We show that this class is M\"{o}bius invariant. Isothermic triangulated surfaces can be characterized either in terms of circle patterns or based on conformal equivalence of triangle meshes. This definition generalizes isothermic quadrilateral meshes.
		
		A consequence is a discrete analog of minimal surfaces. Here the Weierstrass data needed to construct a discrete minimal surface consist of a triangulated plane domain and a discrete harmonic function.
	\end{abstract}

\thanks{This research was supported by the DFG Collaborative Research Centre SFB/TRR 109 \emph{Discretization in Geometry and Dynamics}. The first author was partially supported by Berlin Mathematical School and the Croucher Foundation of Hong Kong.}

\date{\today}

\maketitle

\section{Introduction}\label{sec:intro}

Isothermic surfaces are central objects in classical Differential Geometry. They include all surfaces of revolution, quadrics, constant mean curvature surfaces and many other interesting surfaces \cite{Hertrich-Jeromin2003}. In particular, all classes of surfaces that are describable in terms of integrable systems in some way or other seem to be related to isothermic surfaces \cite{Burstall2012,Cieslinski1995}.

A smooth surface in Euclidean space is called isothermic if it admits conformal curvature line parametrization around every point. Note however that there are various characterizations of isothermic surfaces that do not refer to special parametrizations.   

Discrete Differential Geometry lies between Discrete Geometry and Differential Geometry. The geometry of a discrete surface is determined by the positions of a finite number of vertices, such as those of a triangulated surface in Euclidean space. Smooth surfaces in Differential Geometry can be regarded as limits of discrete surfaces by refinement. The goal of Discrete Differential Geometry is to look for mathematical structures on discrete surfaces as rich as their smooth counterparts. It has many applications, for example in computer graphics and architectural design \cite{Pottmann2015}.

The same spirit applied to complex analysis has led to two different definitions of conformality for planar triangular meshes. One of these two is the theory of \emph{circle patterns} \cite{Schramm1997}, where the conformal structure is defined by the intersection angles of neighboring circumcircles. It is motivated by Thurston's circle packings as a discrete analog of holomorphic functions \cite{Rodin1987}. Another version of discrete conformality is based on {\em conformal equivalence of triangle meshes} \cite{Luo2004,Springborn2008}, where the conformal structure is defined by the length cross ratios of neighboring triangles. Luo introduced this notion when studying a discrete Yamabe flow. Its relation to ideal hyperbolic polyhedra was investigated in \cite{Bobenko2010}. 

Previous definitions of discrete isothermic surfaces were all based on quadrilateral meshes that provide a discrete version of conformal curvature line parametrizations of isothermic surface \cite{Bobenko1996,Bobenko1999,Bobenko2007}. Inspired by discrete integrable systems \cite{Bobenko2008}, Bobenko and Pinkall \cite{Bobenko1996} considered quadrilateral meshes with factorized real cross ratios, which led to further investigation of discrete minimal surfaces and constant mean curvature surfaces \cite{Hertrich1999}. Recently, the notion of curvature was introduced to discrete surfaces with vertex normals \cite{Bobenko2010a,Hoffmann2014}.

Here we aim for a definition of isothermic triangulated surfaces which does not involve conformal curvature line parametrizations. It is motivated by a known (although not well-known) characterization that a smooth surface in Euclidean space is isothermic if and only if locally it admits a nontrivial infinitesimal isometric deformation preserving the mean curvature. The only reference that we could find is from Cie{\'s}li{\'n}ski et al. \cite{Cieslinski1995}, stating that this theorem was known in the 19th century.

Infinitesimal isometric deformations of triangulated surfaces have been extensively studied since Cauchy's rigidity theorem of convex polyhedral surfaces \cite{Whiteley1997,Connelly1993}. An \emph{infinitesimal deformation} of a triangulated surface in space is an assignment of velocity vectors to all the vertices. We can then calculate the change of edge lengths. An infinitesimal deformation is called \emph{isometric} if the edge lengths are preserved.

Suppose we have a realization $f:V\to \field{R}^3$ of a triangulated surface $M=(V,E,F)$ such that each face of $f$ spans an affine plane. Given an infinitesimal isometric deformation $\dot{f}:V \to \field{R}^3$, each triangular face $\{i\!jk\}$ rotates with an angular velocity given by a certain vector $Z_{i\!jk} \in \field{R}^3$. These vectors satisfy a compatibility condition on every interior edge $\{i\!j\}$:
\begin{equation} \label{eq:dihe}
	d\!\dot{f}(e_{i\!j})=d\!f(e_{i\!j}) \times Z_{i\!jk} = d\!f(e_{i\!j}) \times Z_{\!jil},
\end{equation}
where $\{i\!jk\} \in F$ is the left face of $e_{i\!j}$ and $\{\!jil\}\in F$ is the right face.

On the other hand, it is well-known that the integral $\int H\,dA$ of the mean curvature has a very canonical discrete analogue $\sum H_{i\!j}$. Here we have defined the {\em mean curvature} associated to edge $\{i\!j\}$ as 
\[
H_{i\!j} := \alpha_{i\!j} |d\!f(e_{i\!j})|
\]
where $\alpha_{i\!j}$ is the dihedral angle at the edge $\{i\!j\}$ \cite{Sullivan2008}. Under the infinitesimal isometric deformation given by $Z$ on faces (Equation \eqref{eq:dihe}), we have
\[
\dot{H}_{i\!j} =  \dot{\alpha}_{i\!j} |d\!f(e_{i\!j})| = \langle d\!f(e_{i\!j}), Z_{i\!jk} - Z_{\!jil} \rangle.
\]
If we further demanded $\dot{H}_{i\!j} = 0$ on every edge $\{i\!j\}$ then the infinitesimal isometric deformation would be trivial, i.e. an infinitesimal Euclidean deformation. Hence we consider instead the change of the \emph{integrated mean curvature around vertices}
\[
\dot{H}_{i} := \sum_{\!j} \dot{\alpha}_{i\!j} |d\!f(e_{i\!j})| = \sum_{j} \langle d\!f(e_{i\!j}), Z_{i\!jk} - Z_{\!jil} \rangle.
\] 

We are now ready to define isothermic triangulated surfaces. The smooth counterpart of the following formulation for isothermic surfaces is given by Smyth \cite{Smyth2004}. 

\begin{definition}\label{def:isoth}
	A non-degenerate realization $f:V \to \field{R}^3$ of an oriented triangulated surface, with or without boundary, is called \emph{isothermic} if there exists a $\field{R}^3$-valued dual 1-form $\tau: \vv{E}^*_{int} \to \field{R}^3$, not identically zero, such that
	\begin{align}
		\sum_{j} \tau(e^*_{i\!j}) &= 0 \quad \forall i \in V_{int} \label{eq:closed} \\
		d\!f(e_{i\!j}) \times \tau(e^*_{i\!j}) &= 0  \quad \forall \{i\!j\} \in E_{int} \label{eq:cross}\\
		\sum_{j} \langle d\!f(e_{i\!j}), \tau(e^*_{i\!j}) \rangle &=0 \quad \forall i \in V_{int} \label{eq:sum}
	\end{align}
	Here $\vv{E}^*_{int}$ and $V_{int}$ denote the set of interior oriented dual edges and the set of interior vertices of $M$.
\end{definition}
The following is an immediate consequence of our definition.
\begin{corollary}\label{cor:infiniterigidmean}
	A strongly non-degenerate realization of a simply connected triangulated surface is isothermic if and only if there exists an infinitesimal isometric deformation that preserves the integrated mean curvature around vertices but is not induced from Euclidean transformations. 
\end{corollary}

We will state several results about isothermic triangulated surfaces that closely reflect known theorems from the smooth theory. In Section \ref{sec:mob}, \ref{sec:infcon} and \ref{sec:intangles}, we prove
\begin{theorem}\label{thm:mob}
	The class of isothermic triangulated surfaces is M\"{o}bius invariant.
\end{theorem}
\begin{theorem}\label{thm:conformal}
	For a non-degenerate realization $f:V \to \field{R}^3$ of a closed genus-$g$ triangulated surface the space of infinitesimal conformal deformations is of dimension greater or equal to $|V|-6g+6$. The inequality is strict if and only if $f$ is isothermic.
\end{theorem}
\begin{theorem}\label{thm:intangles}
	Suppose $f:V \to \mathbb{R}^3$ is a non-degenerate realization of a simply connected triangulated surface. Then $f$ is isothermic if and only if there exists an infinitesimal deformation that preserves the intersection angles of neighboring circumcircles and neighboring circumspheres but is not induced from M\"{o}bius transformations.
\end{theorem}
Note that Theorem \ref{thm:conformal} concerns the theory of conformal equivalence of triangle meshes \cite{Luo2004,Springborn2008} while Theorem \ref{thm:intangles} deals with the notion of circle patterns \cite{Schramm1997}.

In Section \ref{sec:quad} we show that our definition generalizes isothermic quadrilateral surfaces \cite{Bobenko1996}: Subdividing any isothermic quadrilateral surface in an arbitrary way we obtain an isothermic triangulated surface.

In Sections \ref{sec:cylinder}, \ref{sec:harm} and \ref{sec:inscribedmesh} we provide examples of isothermic triangulated surfaces that are not obtained via quadrilateral isothermic surfaces. Triangulated cylinders generated by discrete groups as well as certain planar triangular meshes and triangulated surfaces inscribed in a sphere are isothermic. 

In Section $\ref{sec:minimalsurf}$ we introduce discrete minimal surfaces via a discrete analogue of the Christoffel duality. Our discrete minimal surfaces are obtained as the reciprocal-parallel meshes for triangulated surfaces (with boundary) inscribed in the unit sphere. This approach mirrors the property that a smooth minimal surface is a Christoffel dual of its Gau{\ss} map. The Weierstrass data needed to construct a discrete minimal surface consist of a planar triangular mesh and a discrete harmonic function. Such harmonic functions were first introduced by NcNeal \cite{McNeal1946}. They were used in linear discrete complex analysis since Duffin \cite{Duffin1956} and have applications in statistical mechanics (see Smirnov \cite{Smirnov2010}).

In Section \ref{sec:smooth}, we review the smooth theory and prove some new theorems that are similar to discrete results established in earlier sections.

Throughout we use the language of discrete differential forms and quaternionic analysis as introduced by Desbrun et al. \cite{Desbrun2006a} and Pedit and Pinkall \cite{Pedit1998}.

\section{Notations} \label{sec:notations}
\begin{definition}
	A triangulated surface $M=(V,E,F)$ is a finite simplicial complex whose underlying topological space is a connected 2-manifold with boundary. The set of vertices (0-cells), edges (1-cells) and triangles (2-cells) are denoted as $V$, $E$ and $F$. 
\end{definition}

Without further notice we assume that all triangulated surfaces under consideration are oriented.

\begin{definition}
	A \emph{non-degenerate} realization of a triangulated surface $M$ in $\field{R}^3$ is a map $f:V \to \field{R}^3$ which is linear on each face and $f_i \neq f_j$ for every edge $\{i\!j\}\in E$. We say $f$ is \emph{strongly non-degenerate} if every face of $f$ spans an affine 2-plane.
\end{definition}

We denote $V_{int}$ and $E_{int}$ the set of interior vertices and the set of interior edges respectively. We write $e_{i\!j}$ as the oriented edge from the vertex $i$ to the vertex $j$. Note that $e_{i\!j} \neq e_{\!ji}$. The set of oriented edges is denoted by $\vv{E}$. The set of interior oriented edges is indicated by $\vv{E}_{int}$.

We recall some notions about discrete differential forms \cite{Desbrun2006a}. A (primal) \emph{1-form} $\omega: \vv{E} \to \field{R}$ is a function defined on oriented edges of $M$ such that
\[
\omega(e_{i\!j})=-\omega(e_{\!ji}).
\]
A 1-form $\omega$ is \emph{closed} if for every face $\{i\!jk\} \in F$
\[
\omega(e_{i\!j})+ \omega(e_{\!jk}) +\omega(e_{ki}) =0.
\] 
It is \emph{exact} if there exists $f:V \to \field{R}$ such that
\[
d\!f(e_{i\!j}) := f_{\!j} - f_i = \omega(e_{i\!j}).
\]
It is easy to check that exactness implies closedness while the converse holds if the discrete surface is simply connected.

Similarly we consider a 1-form $\tau: \vv{E}^*_{int} \to \field{R}$ on the dual cell decomposition $M^*=(V^*,E^*,F^*)$ of $M$ and call $\tau$ a \emph{dual 1-form} on $M$. Here we denote $e_{i\!j}^*$ the dual edge oriented from the right face of $e_{i\!j}$ to the left face. The following notions are natural if we think of a dual 1-form on $M$ as a 1-form on $M^*$. 
A dual 1-form $\tau$ is \emph{closed} if for ever interior vertex $i \in V$ 
\[
\sum_j \tau(e^*_{i\!j}) = 0.
\] 
It is exact if there exists $Z:F \to \field{R}$ such that
\[
dZ(e^*_{i\!j}):= Z_{i\!jk}-Z_{\!jil} = \tau(e^*_{i\!j})
\] 
where $\{i\!jk\}$ denotes the left face of $e_{i\!j}$ and $\{\!jil\}$ denotes the right face.

We distinguish dual 1-forms from primal 1-forms for the following reasons. Firstly, the closedness conditions are different. The closedness conditions are imposed on faces for primal 1-forms while they are imposed at vertices for dual 1-forms. Secondly, a discrete notion of the Hodge star operator is needed to identify 1-forms with dual 1-forms, although it is not explicitly used in this paper. In Discrete Exterior Calculus \cite{Desbrun2006a} one often uses the Hodge star operator, which maps a primal 1-form $\omega$ to a dual 1-form $*\omega$ via
\[
*\omega(e_{i\!j}^*) := (\cot \beta_{i\!j}^{k}+\cot \beta_{i\!j}^{l}) \omega((e_{i\!j}^*)^*) = -(\cot \beta_{i\!j}^{k}+\cot \beta_{i\!j}^{l}) \omega(e_{i\!j}) \quad \forall \{i\!j\} \in E
\]
where $\beta_{i\!j}^k \in (-\pi,\pi)$ denotes the angle $\angle jki$ of the triangle $\{i\!jk\}$ with respect to some discrete metric, i.e. an assignment of edge lengths. Given a dual 1-form $\tau$ and a primal 1-form $d\!f$, we will occasionally write 
\[
\tau(e_{i\!j}^*) = k_{ij} d\!f(e_{i\!j}) 
\]
for some $k:E_{int} \to \field{R}$. Here we think of it as $\tau = k' *d\!f$ for some $k':E_{int} \to \field{R}$.  

\section{M\"{o}bius invariance} \label{sec:mob}
In this section we prove that the class of isothermic triangulated surfaces is invariant under M\"{o}bius transformations.

Given a triangulated surface $f:V \to \field{R}^3$ and a M\"{o}bius transformation $\sigma: \field{R}^3 \cup \{ \infty \} \to \field{R}^3 \cup \{ \infty \} $, we define $\sigma\circ f:V \to \field{R}^3$ as the triangulated surface with vertices $(\sigma\circ f)_i := \sigma\circ f_i$. We consider only the M\"{o}bius transformations that do not map any vertex to infinity. 

Taking $\sigma$ to be minus the inversion in the unit sphere, we obtain a triangulated surface 
\[
\sigma \circ f = -\frac{f}{||f||^2}:=f^{-1} .
\]
Later we will identify $\field{R}^3$ with imaginary quaternions, which explains the notation $f^{-1}$. We are going to show that $f$ is isothermic if and only if $f^{-1}$ is isothermic. We first rewrite the equations from Definition \ref{def:isoth}.

\begin{lemma} \label{lem:eqdef}
	Given a non-degenerate realization $f:V \to \field{R}^3$ of a triangulated surface, a $\field{R}^3$-valued dual 1-form $\tau: \vv{E}^*_{int} \to \field{R}^3$ satisfies
	\begin{align*}
		\sum_{j} \tau(e^*_{i\!j}) &= 0 \quad \forall i \in V_{int}, \\
		d\!f(e_{i\!j}) \times \tau(e^*_{i\!j}) &= 0  \quad \forall \{i\!j\} \in E_{int}, \\
		\sum_{j} \langle d\!f(e_{i\!j}), \tau(e^*_{i\!j}) \rangle &=0 \quad \forall i \in V_{int} 
	\end{align*}
	if and only if there exists $k:E_{int} \to \field{R}$ such that
	\begin{align*}
		k_{i\!j} d\!f(e_{i\!j}) &= \tau(e^*_{i\!j}) \quad \forall \{i\!j\} \in E_{int}, \\
		\sum_{j} k_{i\!j} d\!f(e_{i\!j}) &=0 \quad  \forall i \in V_{int}, \\
		\sum_{j} k_{i\!j} (|f_j|^2 - |f_i|^2) &=0 \quad \forall i \in V_{int}.
	\end{align*}
\end{lemma}
\begin{proof}
	Suppose $k:E_{int} \to \field{R}$ satisfies for every interior vertex $i$
	\[
	\sum_{j} k_{i\!j} d\!f(e_{i\!j}) =0.
	\]
	Then, we have the identity
	\begin{align*}
		\sum_{j} \langle d\!f(e_{i\!j}), k_{i\!j} d\!f(e_{i\!j}) \rangle
		= \sum_{j} k_{i\!j} (|f_j|^2 - |f_i|^2 -2 \langle f_j - f_i, f_i \rangle)
		= \sum_{j} k_{i\!j} (|f_j|^2 - |f_i|^2).
	\end{align*}
	Using this it is easy to verify all our claims. 
\end{proof}
With the above lemma, we can show $f$ is isothermic if and only if $f^{-1}$ is isothermic.
\begin{lemma}\label{lem:trans}
	Suppose a non-degenerate realization $f:V \to \field{R}^3$ of a triangulated surface is isothermic with a non-trivial dual 1-form $\tau$ satisfying Definition \ref{def:isoth}. We write 
	\[
	\tau(e^*_{i\!j}) = k_{i\!j} d\!f(e_{i\!j})
	\]	
	for some $k: E_{int} \to \field{R}$. Then, the triangulated surface $f^{-1}:V \to  \field{R}^3$ is isothermic with corresponding dual 1-form
	\[
	\tilde{\tau}(e^*_{i\!j}) := k_{i\!j} |f_i|^2 |f_j|^2 d\!f^{-1}(e_{i\!j}).
	\]
\end{lemma}
\begin{proof}
	We check that $\tilde{\tau}$ satisfies the equations in Definition \ref{def:isoth} by applying the previous lemma. Firstly for every interior vertex $i$
	\begin{align*}
		\sum_{j} \tilde{\tau}(e^*_{i\!j}) =& \sum_{j} k_{i\!j} |f_i|^2 |f_{\!j}|^2 d\!f^{-1}(e_{i\!j}) \\
		=& \sum_{j} \big(k_{i\!j} |f_i|^2 |f_{\!j}|^2 \frac{f_i}{|f_i|^2} - k_{i\!j} |f_i|^2 f_i + k_{i\!j} |f_i|^2 f_i  - k_{i\!j} |f_i|^2 |f_{\!j}|^2 \frac{f_{\!j}}{|f_{\!j}|^2} \big) \\
		=& f_i \sum_{j} k_{i\!j} (|f_{\!j}|^2 - |f_i|^2) \, + |f_i|^2 \sum_{j} k_{i\!j} (f_i - f_{\!j}) \\
		=& 0. 
	\end{align*}
	Secondly, for every interior vertex $i$
	\begin{align*}
		\sum_{j} k_{i\!j}|f_i|^2 |f_{\!j}|^2 (|f^{-1}_{\!j}|^2 - |f^{-1}_i|^2) = \sum_{j} k_{i\!j} (|f_i|^2 - |f_{\!j}|^2) 	= 0.
	\end{align*} 
	Hence, $f^{-1}$ is isothermic with 1-form $\tilde{\tau}$ satisfying Definition \ref{def:isoth}.	
\end{proof}

\begin{proof}[Proof of Theorem \ref{thm:mob}]
	It follows from the previous lemma and the fact that M\"{o}bius transformations are generated by inversions and Euclidean transformations. 
\end{proof}

\begin{remark}\label{rmk:spin}
	The above calculation can be simplified if written in terms of quaternions. Identifying the Euclidean 3-space with the space of purely imaginary quaternions we obtain
	\[
	d\!f^{-1}(e_{i\!j}) = \bar{f}_i^{-1} d\!f(e_{i\!j}) f_{\!j}^{-1} = \bar{f}_{\!j}^{-1} d\!f(e_{i\!j}) f_i^{-1}
	\]
	and
	\[
	\tilde{\tau}(e^*_{i\!j})= f_i \tau(e^*_{i\!j}) \bar{f_{\!j}}= f_{\!j} \tau(e^*_{i\!j}) \bar{f_i}.
	\]
	These two formulas are similar to the smooth case \cite{Richter1997}.
\end{remark}

Lemma \ref{lem:eqdef} provides another characterization of isothermic triangulated surfaces. We consider the light cone
\[
L:=\{ x\in \field{R}^5\,|\, x_1^2 + x_2^2 +x_3^2 +x_4^2 -x_5^2 =0\}.
\]

\begin{corollary}
	Suppose $f:V \to \field{R}^3$ is a non-degenerate realization of a triangulated surface and $k:E_{int} \to \field{R}$ is a function. Then $f$ is isothermic with corresponding dual 1-form $\tau$ defined by 
	\[
	\tau(e_{i\!j}^*) = k_{i\!j} d\!f(e_{i\!j})  \quad \forall \{i\!j\} \in E_{int}
	\]
	if and only if for every interior vertex $i$
	\begin{equation}\label{eq:selfstress}
		\sum_{j} k_{i\!j} d\hat{f}(e_{i\!j}) =0 \quad \forall i \in V_{int}
	\end{equation}
	where $\hat{f}:V \to L \subset \field{R}^5$ is the lift of $f$ to $\field{R}^5$ defined by
	\[
	\hat{f}_i := (f_i, \frac{1-|f_i|^2}{2},\frac{1+|f_i|^2}{2}) \in L \subset \field{R}^5.
	\]
\end{corollary}

A function $k:E_{int} \to \field{R}$ satisfying Equation \eqref{eq:selfstress} is called a \emph{self-stress} of $\hat{f}$.

It is known that the M\"{o}bius geometry of $\field{R}^3 \cup \{ \infty \}$ is a subgeometry of the projective geometry of $\field{R}P^4$. M\"{o}bius transformations of $\field{R}^3 \cup \{ \infty \}$ are represented as projective transformations of $\field{R}P^4$ preserving the quadric defined by the light cone $L$. If two non-degenerate realizations are related by a projective transformation, then the spaces of self-stresses of the two realizations are isomorphic \cite{Izmestiev2008}. Hence, we obtain another proof of Theorem \ref{thm:mob}.

\section{Infinitesimal conformal deformations} \label{sec:infcon}
We consider infinitesimal conformal deformations for a given closed triangulated surface in space. We show that a surface is isothermic if and only if it is a singular point in the space of all surfaces conformally equivalent to the original one.

\subsection{Conformal equivalence of triangle meshes} \label{sec:conformalequi}

We recall that a \emph{discrete metric} of a triangulated surface is a function $\ell:E \to \field{R}_{+}$ satisfying the triangle inequality on every face.
A non-degenerate realization $f:V \to \field{R}^3$ induces a  discrete metric $\ell:E \to \field{R}_{+}$ via
\begin{align*}
	\quad \ell_{i\!j} := |f_{\!j} - f_i| \quad \forall \{i\!j\} \in E.
\end{align*}

	\begin{definition}[\cite{Luo2004}]
		Two discrete metrics $\ell,\tilde{\ell}:E \to \field{R}_{+}$ on a triangulated surface $M$ are conformally equivalent if there exists $u:V\rightarrow \field{R}$ such that for every edge $\{i\!j\}$
		\[
		\tilde{\ell}_{i\!j}=e^{\frac{u_{i}+u_{j}}{2}}\ell_{i\!j}.
		\]
		Two non-degenerate realizations $f,\tilde{f}:V \to \field{R}^3$ are conformally equivalent if their induced discrete metrics are conformally equivalent.
	\end{definition}
It leads naturally to an infinitesimal version of conformal deformations.
\begin{definition}
	An infinitesimal deformation of a non-degenerate triangulated surface $f:V \to \field{R}^3$ is a map $\dot{f}: V \to \field{R}^3$. It is \emph{conformal} if there exists $u:V \to \field{R}$ such that the change of the induced discrete metric $\dot{\ell}:E\to \field{R}$ satisfies for every edge $\{ij\}$
	\[
	\dot{\ell}_{i\!j} = \frac{u_i+u_j}{2} \ell_{i\!j}. 
	\]
	In particular, $\dot{f}$ is an infinitesimal isometric deformation if $u \equiv 0$.
\end{definition} 

The conformal equivalence class of a triangulated surface in Euclidean space is M\"{o}bius invariant \cite{Bobenko2010}. It can be distinguished via logarithmic length cross ratios.

\begin{definition}
	Given a discrete metric $\ell:E \to \field{R}_{+}$ on a triangulated surface, its \emph{logarithmic length cross ratio} $\log \lcrs : \field{R}^{|E|} \to \field{R}^{|E_{int}|}$ is defined by
	\[
	\log \lcrs(\ell)_{i\!j} := \log \ell_{\!jk} -\log \ell_{ki}+\log \ell_{il}-\log \ell_{l\!j} \quad \forall \{i\!j\}\in E_{int}
	\]
	where $\{i\!jk\}$ is the left face of $e_{i\!j}$ and $\{\!jil\}$ is the right face.
\end{definition}

	\begin{theorem}[\cite{Bobenko2010}]
		Two discrete metrics $\ell$ and $\tilde{\ell}$ on a triangulated surface are conformally equivalent if and only if
		\[
		\log \lcrs(\ell) \equiv \log \lcrs(\tilde{\ell}).
		\]
	\end{theorem}

	\begin{corollary}[\cite{Bobenko2010}]
		The dimension of the space of the conformal equivalence classes of a triangulated surface is $|E|-|V|$.
	\end{corollary}

\subsection{Infinitesimal deformations}\label{sec:infindeform}

In this section, we consider closed triangulated surfaces. Suppose $\ell:E \to \field{R}_{+}$ is a discrete metric on a closed triangulated surface. We consider an infinitesimal change of the discrete metric $\dot{\ell}$ and write it as $\dot{\ell}= \sigma \ell$ for some infinitesimal scaling $\sigma:E \to \field{R}$. Then the change of logarithmic length cross ratio on edge $\{ij\}$ is given by
\[
(\log \lcrs(\ell))^{\LargerCdot}_{i\!j}= \sigma_{jk} - \sigma_{ki} + \sigma_{il} - \sigma_{l\!j} =: L(\sigma)_{i\!j}.
\]
The image of the linear map $L:\field{R}^{|E|} \to \field{R}^{|E|}$ is the tangent space of the space of conformal equivalence classes (which is the same space at all discrete metrics).

\begin{lemma}
	Given a closed triangulated surface. The operator $L$ is skew adjoint with respect to the standard product $(\phantom{a},\phantom{a})$ on $\field{R}^{|E|}$ given by $(a,b):=\sum_{\{i\!j\} \in E} a_{i\!j}b_{i\!j}$ for any $ a,b \in \field{R}^{|E|}$.
\end{lemma}

\begin{proof}
	Let $\delta^{i\!j}: E \to \field{R}$ be the function defined by $(\delta^{i\!j})_{i\!j}=1$ on edge $\{i\!j\}$ and zero on other edges. Then for any $b \in \field{R}^{|E|}$, we have
	\begin{align*}
		L^*(b)_{i\!j} = (\delta^{i\!j},L^*(b))
		= (L(\delta^{i\!j}),b)
		= -b_{\!jk} + b_{ki} - b_{il} + b_{l\!j} 
		= -L(b)_{i\!j}. 
	\end{align*} 
	Thus we have $L^*=-L$. 
\end{proof}

The above lemma implies that we have an orthogonal decomposition
\[
\field{R}^{|E|} = \Ker(L) \oplus \Im(L^*) = \Ker(L) \oplus \Im(L). 
\]

\begin{lemma}Given a closed triangulated surface. We have the following.
	\begin{align*}
		\Ker(L) &= \{a:E \to \field{R}\,|\, \exists u \in \field{R}^V \text{ s.t. } a_{i\!j}= u_i + u_j \quad \forall \{i\!j\}\in E\} \\
		\Im(L) &= \{a:E \to \field{R}| \sum_{j} a_{i\!j} =0 \quad \forall i \in V\} 
	\end{align*}
\end{lemma}
\begin{proof}
	It is obvious that 
	\[
	\{a:E \to \field{R}| \exists u \in \field{R}^V \text{ s.t. } a_{i\!j}= u_i + u_j \quad \forall \{i\!j\}\in E\} \subset \Ker(L).
	\]
	Assume $a \in \Ker(L)$. For each face $\triangle_{i\!jk}$ we define
	\begin{equation}
		u_i := \frac{a_{i\!j}+a_{ki}-a_{jk}}{2} \label{eq:u}
	\end{equation}
	Suppose $\tilde{\triangle}_{il\!j}$ is the neighboring triangle sharing the edge $\{i\!j\}$ with $\triangle_{i\!jk}$. Because of $L(a)_{i\!j}=0$ we have
	\begin{align*}
		u_i = \frac{a_{i\!j}+a_{ki}-a_{jk}}{2} = \frac{a_{i\!j}+a_{il}-a_{l\!j}}{2} 	= \tilde{u}_i.
	\end{align*}
	Since the link of each vertex is a disk (although we only need the vertex link to be a fan), Equation \eqref{eq:u} in fact defines a function $u: V \to \field{R}$ such that for any edge $\{i\!j\}$
	\[
	a_{i\!j} = u_i + u_j.
	\]
	Hence
	\[
	\Ker(L) = \{a:E \to \field{R}\,|\, \exists u \in \field{R}^V \text{ s.t. } a_{i\!j}= u_i + u_j \quad \forall \{i\!j\}\in E\}.
	\]
	
	On the other hand, it is obvious that
	\[
	\Im(L) \subset \{a:E \to \field{R}\,|\, \sum_{j} a_{i\!j} =0 \quad \forall i \in V\}. 
	\]
	Since 
	\[
	\rank(L) = |E| - \dim \Ker(L) = |E| - |V|
	\]
	the two vector spaces are indeed the same. 
\end{proof}

Recall that conformal equivalence classes of a triangular mesh are parametrized by logarithmic length cross ratios. By the inverse function theorem the result below implies that by  deforming a non-isothermic surface in space we can reach all  nearby conformal equivalence classes. It is precisely in the case of an isothermic surface that the hypothesis of the inverse function theorem fails to be satisfied. Thus the space of all non-isothermic non-degenerate realizations in a fixed conformal equivalence class is a smooth manifold.

\begin{theorem}
	Suppose $f:V \to \field{R}^3$ is a non-degenerate realization of a closed triangulated surface. Then $f$ is isothermic if and only if there exists a non-trivial element $a \in \Im(L)$ such that
	\[
	(a,L(\sigma))=0 
	\]
	for all infinitesimal scalings $\sigma:E \to \field{R}$ coming from infinitesimal extrinsic deformations in Euclidean space, i.e. for which there exists $\dot{f}: V \to \field{R}^{3}, W:E \to \field{R}^{3}$ such that $d\!\dot{f} = \sigma d\!f + d\!f \times W$. 
\end{theorem}
\begin{proof}
	Suppose $f$ is isothermic with $\tau$ satisfying Definition \ref{def:isoth}. Let $\dot{f}:V \to \field{R}^3$ be an arbitrary infinitesimal deformation and we write $d\!\dot{f}= \sigma d\!f + d\!f \times W$. Since $\tau$ is closed, i.e. $\sum_j \tau(e_{i\!j}^*) = 0 \,\, \forall i \in V$ we have
	\[
	0 = -\sum_{i \in V} \!  \langle \sum_j \! \tau(e_{i\!j}^*), \dot{f}_i \rangle =\sum_{\{i\!j\} \in E} \! \! \! \! \! \langle \tau(e^*_{i\!j}), d\!\dot{f}(e_{i\!j}) \rangle = \sum_{\{i\!j\}} \!  \langle \tau(e^*_{i\!j}), \sigma_{i\!j} d\!f(e_{i\!j}) + d\!f(e_{i\!j}) \times W_{i\!j} \rangle.
	\]
	From $d\!f(e_{i\!j}) \times \tau(e^*_{i\!j}) = 0$ we obtain
	\[
	0 = \sum_{\{i\!j\}\in E} \langle \tau(e^*_{i\!j}), \sigma_{i\!j} d\!f(e_{i\!j}) + d\!f(e_{i\!j}) \times W_{i\!j} \rangle = \sum_{\{i\!j\} \in E} \langle \tau(e^*_{i\!j}), d\!f(e_{i\!j}) \rangle \sigma_{i\!j} .
	\]
	Using
	\[
	\langle \tau(e^*_{i\!j}), d\!f(e_{i\!j}) \rangle = \langle \tau(e^*_{\!ji}), d\!f(e_{\!ji}) \rangle
	\]
	we see that $\langle \tau,d\!f \rangle :E \to \field{R}$ is well defined. Since we know that for every interior vertex $i$ 
	\[ \sum_{j} \langle d\!f(e_{i\!j}), \tau(e^*_{i\!j}) \rangle =0 \] 
	we thus have $\langle \tau, d\!f \rangle \in \Im(L)$. Hence there exists an non-trivial element $a\in \Im(L)$ such that for every edge $\{i\!j\}$
	\[
	L(a)_{i\!j} = -\langle \tau(e^*_{i\!j}),d\!f(e_{i\!j}) \rangle.
	\]
	Because $\dot{f}$ is arbitrary we conclude that
	\[
	0 = (\langle \tau, d\!f \rangle,\sigma) = (-L(a),\sigma) = (a,L(\sigma))
	\]
	for all infinitesimal scaling $\sigma: E \to \field{R}$ coming from infinitesimal extrinsic deformations.
	
	On the other hand, suppose there exists a non-trivial $a \in \Im(L)$ such that
	\[
	(a,L(\sigma))=0 
	\]
	for all infinitesimal scaling $\sigma \in \field{R}^{|E|}$ coming from infinitesimal extrinsic deformations. We define a dual 1-form $\tau:\vv{E}_{int}^* \to \field{R}^3$ via
	\begin{align*}
		d\!f(e_{i\!j}) \times \tau(e_{i\!j}^*) &= 0,  \\
		\langle d\!f(e_{i\!j}), \tau(e^*_{i\!j}) \rangle &= -L(a)_{i\!j} 
	\end{align*}
	for every edge $\{i\!j\}$. Since $\langle d\!f, \tau \rangle \in \Im(L)$, we have 
	\[
	\sum_{j} \langle d\!f(e_{i\!j}), \tau(e^*_{i\!j}) \rangle =0 \quad \forall i \in V.
	\]
	In addition, for any infinitesimal deformation $\dot{f}: V \to \field{R}^3$ we write $d\!\dot{f}= \sigma d\!f + d\!f \times W$ for some $\sigma:E \to \field{R}$ and $W:E\to \field{R}^3$. We obtain
	\begin{align*}
		-\sum_{i \in V} \langle \sum_j \tau(e_{i\!j}^*), \dot{f}_i \rangle = \sum_{{i\!j}} \langle \tau(e_{i\!j}^*), d\!\dot{f}(e_{i\!j}) \rangle
		= \sum_{{i\!j}} \langle \tau(e_{i\!j}^*), d\!f(e_{i\!j}) \rangle \sigma_{i\!j} 
		= (a,L(\sigma))
		= 0. 
	\end{align*}
	Since $\dot{f}$ is arbitrary we conclude that $\tau$ is closed, i.e.
	\[
	\sum_{j} \tau(e^*_{i\!j})=0 \quad \forall i \in V.
	\]
	Hence, $f$ is isothermic with dual 1-form $\tau$. 
\end{proof}
\begin{proof}[Proof of Theorem \ref{thm:conformal}]
	Consider the composition of maps
	\[
	\{\text{infinitesimal deformations in }\field{R}^3\} \xrightarrow{\sigma} \{ \text{infinitesimal scalings} \} \xrightarrow{L} \{\text{change of lcrs} \}.
	\]
	The space of infinitesimal conformal deformations is exactly $\Ker(L\circ \sigma)$. Moreover, we know
	\begin{align*}
		\dim(\Ker(L\circ \sigma))= 3|V| - \rank(L\circ \sigma)
		\geq 3|V| - (|E|-|V|) 
		= |V| - 6g +6. 
	\end{align*}
	Finally we conclude: The inequality is strict $\iff$ $L\circ \sigma$ is not surjective $\iff$ $f$ is isothermic. 
\end{proof}

Since the conformal equivalence classes of a triangle mesh are parametrized by length cross ratios, we can rephrase the previous theorems as follows. 

\begin{corollary}\label{cor:conflength}
	Given a closed triangulated surface, isothermic realizations are precisely the points in the space of all non-degenerate realizations where the map that takes a non-degenerate realization to the conformal equivalence class of its induced metric fails to be a submersion.
\end{corollary}

It is interesting to see how combinatorics affect geometry. It is known that the number of vertices of a closed genus-$g$ triangulated surface satisfies the Heawood bound \cite{Heawoo1890}
\[
|V| \geq \frac{7+\sqrt{1+48g}}{2}.
\]
This condition is known to be sufficient for the existence of a genus-g triangulated surfaces with $|V|$ vertices except for $g=2$. Comparing the Heawood bound with the inequality in Theorem \ref{thm:conformal} we obtain more examples of isothermic surfaces.

\begin{corollary}
	Every non-degenerate realization of a closed triangulated surface with $|V|< 6g+4$ is isothermic.
\end{corollary}

\begin{proof}
	The space of infinitesimal conformal deformations contains all deformations that come from infinitesimal M{\"o}bius transformations. Therefore this space has dimension at least 10 and hence a surface must be isothermic if $10 > |V| - 6g+6$. 
\end{proof}

Some of these surfaces with small number of vertices can be realized in Euclidean space without self-intersection. For example, there are embedded surfaces with $g=2$ and $|V|=10$ as shown in \cite{Hougardy2007}. 

\section{Preserving intersection angles}\label{sec:intangles}

Given a triangulated surface in Euclidean space, every triangle determines a circumscribed circle and every two triangles sharing an edge determines a circumscribed sphere if the vertices are not con-circular. Two circumscribed circles are called neighboring if their corresponding triangles share an edge. We will call two circumscribed spheres neighboring if they have a common vertex.

Intersection angles of circles and spheres are M\"{o}bius invariant. The intersection angles of neighboring circumcircles of a triangulated surface were used to define discrete Willmore functional \cite{Bobenko2005}.  

\begin{proof}[Proof of Theorem \ref{thm:intangles}]
	Suppose we have an infinitesimal deformation $\dot{f}$ that preserves the angles between circumcircles and circumspheres but is not induced from M\"{o}bius transformations. Then it cannot be that $\dot{f}$ also preserves the length cross ratios (because it is not hard to see that in this case $\dot{f}$ is an infinitesimal M{\"o}bius transformation).
	We write $d\!\dot{f} = \sigma d\!f + d\!f \times W$ for some $\sigma:E \to \mathbb{R}$ and $W:E \to \mathbb{R}^3$. Then the change of logarithmic length cross ratios is $L(\sigma)$ where
	\[
	L(\sigma)_{i\!j}:=\sigma_{jk} -\sigma_{ki} + \sigma_{il}-\sigma_{l\!j} \quad \forall \{i\!j\} \in E_{int}.
	\]
	(See Figure \ref{fig:orientation}.) By our assumptions $L(\sigma)$ does not vanish identically.
	
	We define a dual $1$-form
	\[
	\tau(e^*_{i\!j}):= L(\sigma)_{i\!j} \frac{d\!f(e_{i\!j})}{|d\!f(e_{i\!j})|^2}.
	\] 
	Then we have
	\begin{gather*}
		\tau(e^*_{i\!j}) \times d\!f(e_{i\!j}) =0 \quad \forall \{ i\!j \} \in E\\
		\sum_{j} \langle \tau(e^*_{i\!j}), d\!f(e_{i\!j}) \rangle = \sum_{j} L(\sigma)_{i\!j} = 0 \quad \forall i \in V_{int}.
	\end{gather*}
	In order to show that $f$ is isothermic, we need to verify the closedness of $\tau$, i.e. for every interior vertex i
	\[
	\sum_{j} \tau(e^*_{i\!j}) = 0.
	\]
	
	We identify Euclidean space $\field{R}^3$ with the space $\Im\field{H}$ of imaginary quaternions (Section \ref{sec:smooth}). Pick any vertex $v_0$ and denote its neighboring vertices by $v_1,v_2,\dots,v_n$. Then we take an inversion in the unit sphere centered at $f_0:=f(v_0)$ and denote the images of the neighboring vertices by $\tilde{f}_i$. We have the following relations:
	\begin{align*}
		\tilde{f}_j - f_0 &= (f_j-f_0)^{-1}, \\
		\tilde{f}_{j+1} - \tilde{f}_{j} &= -(f_j-f_0)^{-1} ((f_{j+1}-f_0)-(f_{j}-f_0)) (f_{j+1}-f_0)^{-1}\\ &= -(f_j-f_0)^{-1} (f_{j+1}-f_{j}) (f_{j+1}-f_0)^{-1}.
	\end{align*}
	We define the infinitesimal scaling
	\[
	\tilde{\sigma}_{j,j+1} :=\frac{ |\tilde{f}_{j+1}-\tilde{f}_{j}|^{\LargerCdot}}{|\tilde{f}_{j+1}-\tilde{f}_{j}|}
	\]
	By taking the logarithmic derivative of the following equation 
	\[
	\frac{|\tilde{f}_{j+1}-\tilde{f}_{j}|}{|\tilde{f}_{j}-\tilde{f}_{j-1}|} = \frac{|f_{j+1}-f_{j}||f_{j-1}-f_{0}|}{|f_{j+1}-f_{0}||f_{j}-f_{j-1}|}
	\]
	we obtain for $j=1,\dots,n$
	\begin{equation*}
		\tilde{\sigma}_{j,j+1} - \tilde{\sigma}_{j-1,j} = (L(\sigma))_{0j}
	\end{equation*}
	where $\sigma_{i\!j} = |f_j - f_i|^{\LargerCdot}/|f_j - f_i|$.
	On the other hand, the vertices $\tilde{f}_1,\tilde{f}_2,\dots,\tilde{f}_n,\tilde{f}_1$ form a closed polygon in $\mathbb{R}^3$. We define
	\begin{gather*}
		\tilde{\ell}_{j,j+1}:= |\tilde{f}_{j+1}-\tilde{f}_j|, \\
		\tilde{T}_{j,j+1}:= \frac{\tilde{f}_{j+1}-\tilde{f}_j}{|\tilde{f}_{j+1}-\tilde{f}_j|}.
	\end{gather*}
	Since the polygon is closed, we have
	\[
	0= \sum_{j=1}^{n} \tilde{\ell}_{j,j+1} \tilde{T}_{j,j+1}.
	\]
	The fact that the deformation $\dot{f}$ preserves the intersection angles of neighboring circles and neighboring spheres implies that the angles between the neighboring segments and osculating planes of the closed polygon remain constant. Thus there exists a constant vector $c \in \mathbb{R}^3$ such that
	\begin{align*}
		0 =& \sum \dot{\tilde{\ell}}_{j,j+1} \tilde{T}_{j,j+1} + \sum \tilde{\ell}_{j,j+1} \tilde{T}_{j,j+1} \times c \\ =& \sum \tilde{\sigma}_{j,j+1} (\tilde{f}_{j+1}-\tilde{f}_{j}) \\ 
		=& - \sum L(\sigma)_{0j}(f_j-f_0)^{-1} \\
		=& \sum_{j=1}^n \tau(e^*_{0j}).
	\end{align*}
	To show that the converse is true one only has to reverse the previous argument.     
\end{proof}

\section{Example: Isothermic quadrilateral surfaces} \label{sec:quad}
We show that isothermic quadrilateral surfaces as defined by Bobenko and Pinkall \cite{Bobenko1996} are isothermic under our definition (after an arbitrary subdivision into triangles). Isothermic quadrilateral surfaces are analogous to conformal curvature line parametrizations of smooth isothermic surfaces. They can be treated using the theory of integrable systems. New isothermic surfaces can be obtained from a given isothermic surface via  the Christoffel duality and Darboux transformations \cite{Hertrich1999}. Special discrete surfaces related to isothermic quadrilateral meshes were studied in \cite{Bobenko2014,Bobenko2006}.

Questions about infinitesimal rigidity of quadrilateral meshes have been considered by \cite{Schief2008,Wallner2008}. 

We first review some results on isothermic quadrilateral surfaces from \cite{Bobenko2008}. Then we construct an infinitesimal isometric deformation for every isothermic quadrilateral surface and show that the change of mean curvature around each vertex is zero. In this way we obtain isothermic triangulated surfaces from the earlier notion of isothermic quadrilateral surfaces.
\subsection{Review}\label{sec:review}
	\begin{definition}[\cite{Bobenko1996}]
		A discrete isothermic net is a map $F:\field{Z}^2 \to \field{R}^3$, for which all elementary quadrilaterals have factorized real cross-ratios in the form
		\[
		q(F_{m,n},F_{m+1,n},F_{m+1,n+1},F_{m,n+1}) = \frac{\alpha_m}{\beta_n} \quad \forall m,n \in \field{Z},
		\]
		where $\alpha_m \in \field{R}$ does not depend on $n$ and $\beta_n \in \field{R}$ not depend on $m$.
	\end{definition}
\begin{figure}[ht]
	\centering
	\begin{tikzpicture}[scale=1.5]
	\coordinate [label=below:$F_{m,n}$](mn) at (0,0);
	\coordinate [label=below:$F_{m+1,n}$](m1n) at (1,0);
	\coordinate [label=above:$F_{m+1,n+1}$] (m1n1) at (1,1);
	\coordinate [label=above:$F_{m,n+1}$](mn1) at (0,1);
	\draw (mn)--(m1n)--(m1n1)--(mn1)--(mn);
	\end{tikzpicture}
	\caption{An elementary quadrilateral}
\end{figure}
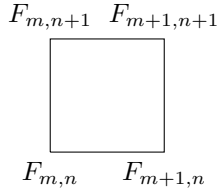
	\begin{theorem}[\cite{Bobenko1996}] \label{thm:dual}
		Let $F:\field{Z}^2 \to \field{R}^3$ be a discrete isothermic net. Then the discrete net $F^*: \field{Z}^2 \to \field{R}^3$ defined (up to translation) by the equations
		\begin{gather*}
			F^*_{m+1,n} - F^*_{m,n} = \alpha_m \frac{F_{m+1,n}-F_{m,n}}{||F_{m+1,n}-F_{m,n}||^2}, \\
			F^*_{m,n+1} - F^*_{m,n} =\beta_n \frac{F_{m,n+1}-F_{m,n}}{||F_{m,n+1}-F_{m,n}||^2}
		\end{gather*}
		is isothermic. $F^*$ is called the Christoffel dual of $F$.
	\end{theorem}
We need a formula for the diagonals of its Christoffel dual (Corollary 4.33 in \cite{Bobenko2008}).
\begin{lemma} \label{lma:dia}
	Given a discrete isothermic net $F$, the diagonals of any elementary quadrilateral of its Christoffel dual are given by
	\begin{gather*}
		F^*_{m+1,n} - F^*_{m,n+1} = (\alpha_m - \beta_n) \frac{F_{m+1,n+1}-F_{m,n}}{||F_{m+1,n+1}-F_{m,n}||^2}, \\
		F^*_{m+1,n+1} - F^*_{m,n} = (\alpha_m - \beta_n) \frac{F_{m+1,n}-F_{m,n+1}}{||F_{m+1,n}-F_{m,n+1}||^2}. 
	\end{gather*}
\end{lemma}
\subsection{Infinitesimal flexibility of isothermic quadrilateral surfaces}\label{sec:isothermicnet}
Given a discrete isothermic net we first arbitrarily introduce a diagonal for each quadrilateral in order to get a triangulation. Then, we define infinitesimal rotations on faces as follows.
\paragraph{Rule:} Suppose $ABCD$ is an elementary quadrilateral of a discrete isothermic net $F: \field{Z}^2 \to \field{R}^3$ and the diagonal $AC$ is inserted. Then we get two triangles $ABC$ and $ACD$. We define infinitesimal rotations $Z_{ABC}:= B^*$ and $Z_{ACD}:= D^*$ where $B^*$ and $D^*$ are the corresponding vertices of the Christoffel dual $F^*: \field{Z}^2 \to \field{R}^3$. 
\begin{theorem}
	Suppose we are given a discrete isothermic net $F:\field{Z}^2 \to \field{R}^3$ and its Christoffel dual $F^*:\field{Z}^2 \to \field{R}^3$. We assume that the faces of $F$ have been subdivided into triangles in an arbitrary way. Then the infinitesimal rotations given by the above rule for each triangle define an infinitesimal isometric deformation of the triangulated surface. Moreover, the infinitesimal deformation preserves the integrated mean curvature and is not induced from Euclidean transformations.  
\end{theorem}
\begin{proof}
	By Theorem \ref{thm:dual} and Lemma \ref{lma:dia}, the infinitesimal rotations of two adjacent triangles are compatible on the common edge. Therefore they define an infinitesimal isometric deformation.
	
	It remains to show that around an arbitrary vertex the change of the integrated mean curvature around vertices is zero. For every vertex there are $2^4=16$ ways of inserting diagonals on the four neighboring quadrilaterals. Taking the symmetry into account we can reduce them to $6$ cases. We enumerate these $6$ cases and calculate the change of mean curvature on each edge in Table \ref{table}. It can be checked directly that in all cases the sum around the vertex is zero.   
	\begin{table}[t]
		\centering
		\begin{tabular}{ccc}
			\begin{tikzpicture}[scale=1.6]
			\coordinate [label=below right:$F_{m,n}$](mn) at (0,0);
			\coordinate [label=right:$0\phantom{\alpha_{m}}$](m1n) at (1,0);
			\coordinate (m1n1) at (1,1);
			\coordinate [label=above:$0$](mn1) at (0,1);
			\coordinate (m-1n1) at (-1,1);
			\coordinate [label=left:$\phantom{\alpha_{m-1}}0$](m-1n) at (-1,0);
			\coordinate (m-1n-1) at (-1,-1);
			\coordinate [label=below:$0$](mn-1) at (0,-1);
			\coordinate (m1n-1) at (1,-1);
			\draw[dashed] (m1n)--(m1n1)--(mn1)--(m-1n1)--(m-1n)--(m-1n-1)--(mn-1)--(m1n-1)--(m1n);
			\draw (mn)--(m1n);
			\draw (mn)--(mn1);
			\draw (mn)--(m-1n);
			\draw (mn)--(mn-1);
			
			\draw[dashed] (m1n)--(mn1)--(m-1n)--(mn-1)--(m1n);
			
			\end{tikzpicture}
			
			&   &
			
			\begin{tikzpicture}[scale=1.6]
			\coordinate [label=below right:$F_{m,n}$] (mn) at (0,0);
			\coordinate [label=right:$\alpha_m$] (m1n) at (1,0);
			\coordinate [label=above :$-\alpha_m + \beta_n$] (m1n1) at (1,1);
			\coordinate [label=above:$-\beta_n$] (mn1) at (0,1);
			\coordinate (m-1n1) at (-1,1);
			\coordinate [label=left:$\phantom{\alpha_{m-1}}0$](m-1n) at (-1,0);
			\coordinate (m-1n-1) at (-1,-1);
			\coordinate [label=below:$0$](mn-1) at (0,-1);
			\coordinate (m1n-1) at (1,-1);
			\draw[dashed] (m1n)--(m1n1)--(mn1)--(m-1n1)--(m-1n)--(m-1n-1)--(mn-1)--(m1n-1)--(m1n);
			\draw (mn)--(m1n);
			\draw (mn)--(mn1);
			\draw (mn)--(m-1n);
			\draw (mn)--(mn-1);
			
			\draw[dashed] (mn1)--(m-1n)--(mn-1)--(m1n);
			\draw (mn)--(m1n1);
			
			\end{tikzpicture}
			
			\\
			\begin{tikzpicture}[scale=1.6]
			\coordinate [label=below right:$\quad F_{m,n}$] (mn) at (0,0);
			\coordinate [label=right:$0$] (m1n) at (1,0);
			\coordinate [label=above :$-\alpha_m + \beta_n$] (m1n1) at (1,1);
			\coordinate [label=above:$-\beta_n$] (mn1) at (0,1);
			\coordinate (m-1n1) at (-1,1);
			\coordinate [label=left:$\phantom{\alpha_{m-1}}0$] (m-1n) at (-1,0);
			\coordinate (m-1n-1) at (-1,-1);
			\coordinate [label=below:$-\beta_{n-1}$] (mn-1) at (0,-1);
			\coordinate [label=below :$\alpha_m + \beta_{n-1}$] (m1n-1) at (1,-1);
			\draw[dashed] (m1n)--(m1n1)--(mn1)--(m-1n1)--(m-1n)--(m-1n-1)--(mn-1)--(m1n-1)--(m1n);
			\draw (mn)--(m1n);
			\draw (mn)--(mn1);
			\draw (mn)--(m-1n);
			\draw (mn)--(mn-1);
			
			\draw[dashed] (mn1)--(m-1n)--(mn-1);
			\draw (mn)--(m1n1);
			\draw (mn)--(m1n-1);
			
			\end{tikzpicture}
			
			&  &
			
			\begin{tikzpicture}[scale=1.6]
			\coordinate [label=below right:$F_{m,n}$] (mn) at (0,0);
			\coordinate [label=right:$\alpha_m$](m1n) at (1,0);
			\coordinate [label=above :$-\alpha_m + \beta_n$] (m1n1) at (1,1);
			\coordinate [label=above:$-\beta_n$] (mn1) at (0,1);
			\coordinate (m-1n1) at (-1,1);
			\coordinate [label=left: $\alpha_{m-1}$](m-1n) at (-1,0);
			\coordinate [label=below:$-\alpha_{m-1}+\beta_{n-1}$] (m-1n-1) at (-1,-1);
			\coordinate [label=below right:$-\beta_{n-1}$] (mn-1) at (0,-1);
			\coordinate (m1n-1) at (1,-1);
			\draw[dashed] (m1n)--(m1n1)--(mn1)--(m-1n1)--(m-1n)--(m-1n-1)--(mn-1)--(m1n-1)--(m1n);
			\draw (mn)--(m1n);
			\draw (mn)--(mn1);
			\draw (mn)--(m-1n);
			\draw (mn)--(mn-1);
			
			\draw[dashed] (mn1)--(m-1n);
			\draw[dashed] (m1n)--(mn-1);
			\draw (mn)--(m1n1);
			\draw (mn)--(m-1n-1);
			
			\end{tikzpicture}
			
			\\
			\begin{tikzpicture}[scale=1.6]
			\coordinate [label=below right:$\quad F_{m,n}$] (mn) at (0,0);
			\coordinate [label=right:$0$] (m1n) at (1,0);
			\coordinate [label=above :$-\alpha_m + \beta_n$] (m1n1) at (1,1);
			\coordinate [label=above:$-\beta_n$] (mn1) at (0,1);
			\coordinate (m-1n1) at (-1,1);
			\coordinate [label=left:$\alpha_{m-1}$] (m-1n) at (-1,0);
			\coordinate [label=below:$-\alpha_{m-1}+\beta_{n-1}$](m-1n-1) at (-1,-1);
			\coordinate [label=below :$0$](mn-1) at (0,-1);
			\coordinate [label=below:$\alpha_m - \beta_{n-1}$] (m1n-1) at (1,-1);
			\draw[dashed] (m1n)--(m1n1)--(mn1)--(m-1n1)--(m-1n)--(m-1n-1)--(mn-1)--(m1n-1)--(m1n);
			\draw (mn)--(m1n);
			\draw (mn)--(mn1);
			\draw (mn)--(m-1n);
			\draw (mn)--(mn-1);
			
			\draw [dashed] (mn1)--(m-1n);
			\draw (mn)--(m1n-1);
			\draw (mn)--(m1n1);
			\draw (mn)--(m-1n-1);
			
			\end{tikzpicture}
			
			&  &
			
			\begin{tikzpicture}[scale=1.6]
			\coordinate [label=below right:$\quad F_{m,n}$] (mn) at (0,0);
			\coordinate [label=right:$0$](m1n) at (1,0);
			\coordinate [label=above :$-\alpha_m + \beta_n$] (m1n1) at (1,1);
			\coordinate [label=above:$0$](mn1) at (0,1);
			\coordinate [label=above: $\alpha_{m-1}-\beta_n$] (m-1n1) at (-1,1);
			\coordinate [label=left:$\phantom{\alpha_{m-1}}0$](m-1n) at (-1,0);
			\coordinate [label=below: $-\alpha_m+\beta_{n-1}$](m-1n-1) at (-1,-1);
			\coordinate [label=below:$0$](mn-1) at (0,-1);
			\coordinate [label=below: $\alpha_m - \beta_{n-1}$](m1n-1) at (1,-1);
			\draw[dashed] (m1n)--(m1n1)--(mn1)--(m-1n1)--(m-1n)--(m-1n-1)--(mn-1)--(m1n-1)--(m1n);
			\draw (mn)--(m1n);
			\draw (mn)--(mn1);
			\draw (mn)--(m-1n);
			\draw (mn)--(mn-1);
			
			\draw (mn)--(m-1n1);
			\draw (mn)--(m1n-1);
			\draw (mn)--(m1n1);
			\draw (mn)--(m-1n-1);
			
			\end{tikzpicture}
			
		\end{tabular}
		\caption{The six types of triangulations around a vertex $F_{m,n}$ and the corresponding change of mean curvature on the edges.}
		\label{table}
	\end{table} 
\end{proof}
\begin{remark}
	Although the infinitesimal rotations on faces depend on the triangulation, the deformations of the edges already present in the quad mesh do not. For example, the change of the edge $F_{m+1,n} - F_{m,n}$ is given by
	\begin{align*}
		(\dot{F}_{m+1,n} - \dot{F}_{m,n}) &= (F_{m+1,n} - F_{m,n}) \times F^*_{m+1,n} \\ &= (F_{m+1,n} - F_{m,n}) \times F^*_{m,n} \\ &= (F_{m+1,n} - F_{m,n}) \times \frac{F^*_{m+1,n} + F^*_{m,n}}{2}.
	\end{align*}
	Here we have used $(F^*_{m+1,n} - F^*_{m,n}) \parallel (F_{m+1,n} - F_{m,n})$. Moreover, the quadrilaterals do not stay con-circular under the infinitesimal deformation.
\end{remark}
The infinitesimal isometric deformation defined above has an exact counterpart in the smooth theory. Given a simply connected isothermic surface $f$ and its Christoffel dual $f^*$, there exists an infinitesimal isometric deformation $\dot{f}$ satisfying $d\!\dot{f}=d\!f \times f^*$ (Section \ref{sec:smooth}). It preserves the mean curvature but does not preserve curvature lines. If in addition the curvature lines were preserved, the shape operator would remain unchanged and the deformation would be trivial, i.e. an infinitesimal Euclidean transformation.

\section{Example: Homogeneous discrete cylinders}\label{sec:cylinder}
In this section we show that every homogeneous triangulation of a circular cylinder in $\mathbb{R}^3$ is isothermic. Here ``homogeneous'' means that there is a subgroup of Euclidean transformations that acts transitively at vertices and respects the combinatorics. Note that in general none of the edges of such an isothermic discrete cylinder is aligned with the curvature line directions of the underlying smooth cylinder.

We consider the group $G$ of all Euclidean motions that fix the $z$-axis. Every element $g \in G$ is of the form that acts on a point $p \in \field{R}^3$ as
\[
g(p) = \left( \begin{array}{ccc} \cos \theta & \sin \theta & 0\\ -\sin \theta & \cos \theta & 0\\
0 & 0 & 1 
\end{array} \right) p + \left( \begin{array}{c} 0 \\ 0 \\ h
\end{array}\right)
\]
where $\theta,h \in \field{R}$.

We pick two elements $g_1,g_2$ of $G$ in general position and consider the group $H$ generated by $g_1,g_2$. For a generic choice of $g_1,g_2$ the group $H$ is isomorphic to $\field{Z}^2$. An element $(s,t) \in \field{Z}^2$ corresponds to the element $g_1^s g_2^t \in H$. 

We also consider $\field{Z}^2$ as the vertex set of a triangulated surface with faces of the form $\{(s,t),(s+1,t),(s,t+1)\}$ or $\{(s+1,t),(s+1,t+1),(s,t+1)\}$ (Figure \ref{fig:z2}).
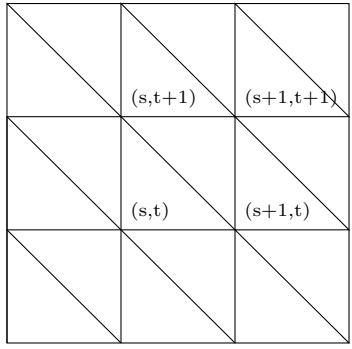
\begin{figure}[ht]
	\centering
	\begin{tikzpicture}[scale=1.5]
	\tikzstyle{every node}=[font=\scriptsize]
	
	\coordinate [label=above right: {(s,t)}] (mn) at (0,0);
	\coordinate [label=above right: {(s+1,t)}] (m1n) at (1,0);
	\coordinate [label=above right: {(s+1,t+1)}] (m1n1) at (1,1);
	\coordinate [label=above right: {(s,t+1)}] (mn1) at (0,1);
	\coordinate  (m-1n1) at (-1,1);
	\coordinate (m-1n) at (-1,0);
	\coordinate (m-1n-1) at (-1,-1);
	\coordinate (mn-1) at (0,-1);
	\coordinate (m1n-1) at (1,-1);
	\coordinate (m2n) at (2,0);
	\coordinate (m2n1) at (2,1);
	\coordinate (m2n2) at (2,2);
	\coordinate (m1n2) at (1,2);
	\coordinate (mn2) at (0,2);
	\coordinate (m-1n2) at (-1,2);
	\coordinate (m2n-1) at (2,-1);
	
	\draw (m2n-1)--(m2n)--(m2n1)--(m2n2)--(m1n2)--(mn2)--(m-1n2);
	\draw (m1n-1)--(m2n-1)--(m1n)--(m2n)--(m1n1)--(m2n1)--(m1n2)--(m1n1)--(mn2)--(mn1)--(m-1n2)--(m-1n-1);
	\draw (m1n)--(m1n1)--(mn1)--(m-1n1)--(m-1n)--(m-1n-1)--(mn-1)--(m1n-1)--(m1n);
	\draw (mn)--(m1n);
	\draw (mn)--(mn1);
	\draw (mn)--(m-1n);
	\draw (mn)--(mn-1);
	
	\draw (mn)--(m-1n1);
	\draw (mn)--(m1n-1);
	\draw (mn1)--(m1n);
	\draw (m-1n)--(mn-1);
	
	\end{tikzpicture}
	\caption{A triangulated surface with vertex set $\field{Z}^2$}
	\label{fig:z2}
\end{figure}

We now define a map $f:\mathbb{Z}^2 \to \field{R}^3$ by picking $r>0$ and setting
\[
f(s,t)= g_1^sg_2^t(r,0,0).
\]
For suitable $g_1,g_2 \in G$ this map $f$ will be a non-degenerate realization. Figure \ref{fig:cylinder} shows a piece of such a discrete surface.

\begin{figure}[ht]
	\centering
	\includegraphics[width=0.8\textwidth]{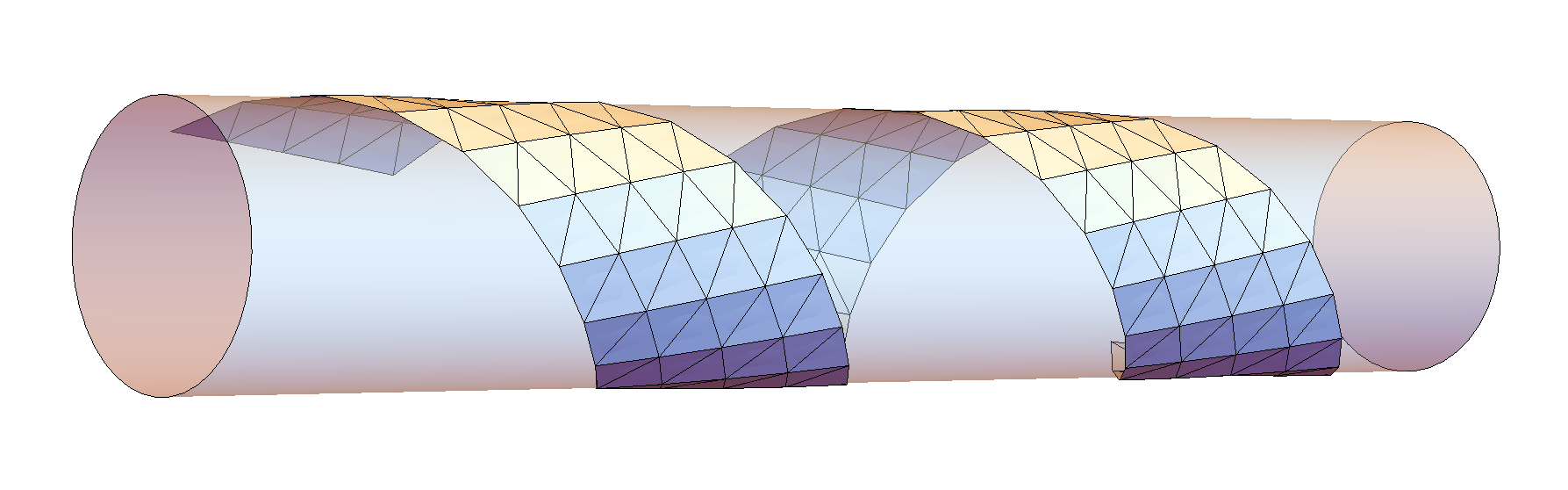}
	\caption{A strip of an isothermic triangulated cylinder}
	\label{fig:cylinder}
\end{figure}

We now prove that realizations $f:\field{Z}^2 \to \field{R}^3$ constructed as above are isothermic. We will do this by showing that they admit a non-trivial infinitesimal isometric deformation preserving the integrated mean curvature. Note that up to symmetry there are only three types of edges, represented by $\{f(0,0),f(1,0)\}$, $\{f(1,0),f(0,1)\}$ and $\{f(0,1),f(0,0)\}$. We denote their lengths by
\[\ell_a (r,\theta_1,h_1,\theta_2,h_2)\quad ,\quad \ell_b (r,\theta_1,h_1,\theta_2,h_2) \quad,\quad \ell_c (r,\theta_1,h_1,\theta_2,h_2).\]
The integrated mean curvature is the same at all vertices. We denote it by \[ H(r,\theta_1,h_1,\theta_2,h_2).\] Now the derivative of the map $\mu:=(\ell_a,\ell_b, \ell_c,H):\mathbb{R}^5 \to \mathbb{R}^4$ has a non-trivial kernel at every point $(r,\theta_1,h_1,\theta_2,h_2)\in\mathbb{R}^5$. Moreover, it is easy to see that any non-zero element
\[(\dot{r},\dot{\theta}_1,\dot{h}_1,\dot{\theta}_2,\dot{h}_2) \in \mbox{ker}\,d\mu \]
corresponds to an infinitesimal deformation of $f$ which is not induced from Euclidean transformations. This infinitesimal deformation preserves all the edge lengths and the integrated mean curvature around vertices. Therefore the triangulated cylinder $f$ is isothermic.

\section{Example: Planar triangular meshes}\label{sec:harm}

In this section we show that certain planar triangular meshes are isothermic. For a simply connected surface, we know that a realization is isothermic if and only if there exists a non-trivial infinitesimal isometric deformation preserving the integrated mean curvature (Corollary \ref{cor:infiniterigidmean}). We will see that every such deformation of a planar triangular mesh is given by a discrete harmonic function in the sense of the cotangent Laplacian  \cite{McNeal1946,Smirnov2010}.   

\begin{theorem}\label{thm:planarisothermic}
	Suppose $f:V \to \field{R}^2 \subset \field{R}^3$ is a non-degenerate realization of a triangulated surface with Euler Characteristic $\chi$ and $|V_b|$ boundary vertices. Then $f$ is isothermic if $ |V_b| - 3\chi > 0$.
\end{theorem}
\begin{proof}
	Since each boundary component is a simple closed polygon, the number of boundary edges is $|E_b|=|V_b|$. The Euler characteristic is given by
	\[
	|V|- |E| +|F|= \chi.
	\]
	Since the surface is triangulated, we have
	\[
	3|F| = 2|E|-|E_b|.
	\]
	Hence
	\begin{align*}
		|E_{int}| - 3|V_{int}| = |V_b| - 3\chi.	
	\end{align*}
	By Lemma \ref{lem:eqdef}, a dual 1-form satisfying Definition \ref{def:isoth} is equivalent to a function $k:E_{int} \to \field{R}^3$ such that
	for every interior vertex $i$
	\begin{align}
		\sum_j k_{ij} d\!f(e_{ij}) &=0  \label{eq:1}\\
		\sum_j k_{ij} (|f_j|^2-|f_i|^2) &=0  \label{eq:2}
	\end{align}
	which is a system of linear equations. By simple counting and using the fact that $f$ is planar we obtain a lower bound for the dimension of the solution space  
	\[
	\dim\{k:E_{int}\to \field{R} \text{ satisfying } \eqref{eq:1} \eqref{eq:2} \} \geq |E_{int}|-2|V_{int}|-|V_{int}| = |V_b| - 3\chi.
	\]
	Hence $f$ is isothermic if $|V_b| - 3\chi >0$. 
\end{proof}
In particular the above theorem implies that every planar triangulated disk ($\chi=1$) with more than 3 boundary vertices is isothermic. Since a disk is simply connected, by Corollary \ref{cor:infiniterigidmean} there exists a non-trivial infinitesimal isometric deformation preserving the integrated mean curvature. The following indicates how to obtain such infinitesimal deformations.
\begin{theorem}
	Let $f:V \to \field{R}^2 \subset \field{R}^3$ be a strongly non-degenerate realization of a triangulated surface with normal $N\in \field{S}^2$ and $u:V \to \field{R}$ be a function. Then the infinitesimal isometric deformation $\dot{f}:=u N$ preserves the integrated mean curvature if and only if $u$ is a discrete harmonic function in the sense of the cotangent Laplacian, i.e. for every interior vertex $i$
	\[
	\sum_j (\cot \beta_{i\!j}^k + \cot \beta_{ji}^l) (u_j - u_i) =0. 
	\]
	where $\beta_{i\!j}^k \in (-\pi,\pi)$ denotes the angle $\angle jki$ of the triangle $\{i\!jk\}$ under the image of $f$ (Figure \ref{fig:orientation}).
	\begin{figure}[h]
		\centering
		\def\svgwidth{0.4\textwidth}
		\resizebox{0.4\textwidth}{!}{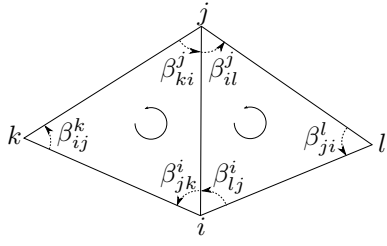}
		\caption{Two neighboring and oriented triangles sharing edge $\{i\!j\}$.}
		\label{fig:orientation}
	\end{figure}
\end{theorem}

\begin{proof}
	It is easy to see that for any function $u: V \to \field{R}$ the infinitesimal deformation $\dot{f} = uN$ preserves edge lengths. The infinitesimal rotation of a face $\{i\!jk\}$ is given by
	\[
	Z_{i\!jk} = -\frac{u_i d\!f(e_{\!jk})+ u_j d\!f(e_{ki})+u_k d\!f(e_{i\!j})}{2 A_{i\!jk}}
	\]
	where $A_{i\!jk}:= \langle (df(e_{i\!j})\times df(e_{ik}))/2,N \rangle$ is the (signed) area of the triangle $\{i\!jk\}$. This follows from
	\[
	d\dot{f}(e_{i\!j}) = u_j N - u_i N = d\!f(e_{i\!j}) \times Z_{i\!jk}.
	\]
	To preserve the integrated mean curvature around vertices, the map $Z: F \to \field{R}^3$ has to satisfy for every interior vertex $i$
	\begin{equation}
		\label{eq:meanzero}
		0 = \sum_{\{i\!j\}\in E:i} \langle d\!f(e_{i\!j}), Z_{i\!jk} - Z_{\!jil} \rangle = \sum_{\{ijk\}\in F:i} \langle d\!f(e_{k\!j}), Z_{i\!jk}\rangle.
	\end{equation}
	Note that
	\begin{align*}
		\langle d\!f(e_{k\!j}), Z_{i\!jk}\rangle =& \langle d\!f(e_{k\!j}), -\frac{u_i d\!f(e_{jk})+ u_j d\!f(e_{ki})+u_k d\!f(e_{i\!j})}{2 A_{i\!jk}}\rangle \\
		=& \langle d\!f(e_{\!jk}), \frac{ (u_j-u_i) d\!f(e_{ki})+(u_k-u_i) d\!f(e_{i\!j})}{2 A_{i\!jk}}\rangle \\
		=& -\cot \beta^k_{ij} (u_j -u_i) - \cot \beta^l_{ji}(u_k-u_i).
	\end{align*}
	Thus \eqref{eq:meanzero} is equivalent to saying that for every interior vertex $i $
	\[
	0 = 
	\sum_{\{i\!j\} \in E:i} (\cot \beta_{i\!j}^k + \cot \beta_{ji}^l) (u_j - u_i).
	\]
	Hence the infinitesimal deformation $uN$ preserves the integrated mean curvature if and only if $u$ is a discrete harmonic function.
\end{proof}
In fact, every infinitesimal isometric deformation of a strongly non-degenerate planar triangular mesh is of the form $u\,N$ for some function $u$ modulo infinitesimal Euclidean motions.

It can be checked that an infinitesimal normal deformation $u\,N$ of a planar mesh $f$ is an Euclidean motion if and only if $u$ is a linear function, i.e. if there exists a vector $a \perp N $ and a constant $c \in \field{R}$ such that
\[
u = \langle a,f \rangle + c.
\]
\section{Example: Inscribed triangular meshes}\label{sec:inscribedmesh}

Since the notion of isothermic triangulated surfaces is M\"{o}bius invariant (Theorem \ref{thm:mob}), our results for planar triangular meshes can be rephrased for triangular meshes inscribed in a sphere. Theorem \ref{thm:planarisothermic} immediately implies the following.
\begin{corollary}
	Suppose $f:V \to \field{S}^2$ is a non-degenerate realization of a triangulated surfaces with Euler Characteristic $\chi$ and $|V_b|$ boundary vertices. Then $f$ is isothermic if $|V_b| - 3\chi >0$.
\end{corollary}  
In particular, it follows that every inscribed triangulated disk with more than 3 boundary vertices is isothermic. This is analogous to the fact that disks immersed smoothly in a sphere are isothermic.

On the other hand, to show that a triangulated surface is isothermic, we can look for a non-trivial infinitesimal isometric deformation preserving the integrated mean curvature. For an inscribed triangulated surface, we will see that it suffices to find an infinitesimal isometric deformation. 

The following lemma shows that for any triangular mesh inscribed in a sphere a dual 1-form $\tau$ satisfying Equation \eqref{eq:closed} and \eqref{eq:cross} in Definition \ref{def:isoth} will satisfy Equation \eqref{eq:sum} automatically.
\begin{lemma}
	Given a non-degenerate realization $f:V \to \field{S}^2$ of a triangulated surface and a function $k:E_{int} \to \field{R}$, then for every interior vertex $i$
	\[
	\sum_j k_{ij} d\!f(e_{i\!j}) =0  \quad \implies \quad \sum_j k_{ij} |d\!f(e_{i\!j})|^2 =0.
	\] 
\end{lemma}
\begin{proof}
	Since $|f|\equiv 1$, we have
	\begin{align*}
		\sum_{j} k_{i\!j} |d\!f(e_{i\!j})|^2 = \sum_{j} k_{i\!j} (2|f_i|^2 - 2\langle f_i,f_{\!j} \rangle) = 2 \langle f_i,\sum_{j}k_{i\!j} (f_i - f_{\!j}) \rangle =0.
	\end{align*} 
\end{proof}
\begin{theorem}\label{thm:sphiso}
	Suppose $f:V \to \field{S}^2$ is a strongly non-degenerate realization of a triangulated surface $M$. Then every infinitesimal isometric deformation preserves the integrated mean curvature. 
	
	Hence if $f$ is infinitesimally flexible, then it is isothermic. If $f$ is isothermic and $M$ is simply connected, then $f$ is infinitesimally flexible.
\end{theorem}
\begin{proof}
	Suppose an infinitesimal isometric deformation is given by a rotation vector field $Z:F \to \field{R}^3$. The compatibility condition implies that there exists $k:E_{int}\to \field{R}$ such that on every interior edge $\{i\!j\}$ 
	\[
	(Z_{i\!jk} - Z_{\!jil}) =k_{ij}d\!f(e_{i\!j})
	\] 
	where $\{i\!jk\}$ denotes the left face of $e_{i\!j}$ and $\{\!jil\}$ denotes the right face. The previous lemma yields for any vertex $i \in V_{int}$,
	\[
	\dot{H}_i = \sum_{j} k_{i\!j} |d\!f(e_{i\!j})|^2 = 0.
	\]
	Hence the integrated mean curvature is preserved. 
\end{proof}

\begin{example}
	Jessen's orthogonal icosahedron is obtained from a regular icosahedron by flipping $6$ edges symmetrically without self intersection \cite{Jessen1967,Goldberg1978}. Its vertices are exactly those of a regular icosahedron and hence lie on a sphere (Figure \ref{fig:jessen}). It is known to be infinitesimally flexible and thus isothermic.
	
	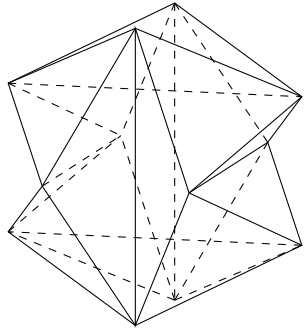
\begin{figure}[ht]
		\centering
		\tdplotsetmaincoords{80}{15}
		\begin{tikzpicture}[scale=1,tdplot_main_coords]
		\coordinate (A) at (1,-2,0); 1
		\coordinate (E) at (-1,-2,0);2
		\coordinate (D) at (2,0,1);3
		\coordinate (H) at (2,0,-1);4
		\coordinate (B) at (0,-1,2);5
		\coordinate (I) at (0,1,2);6
		\coordinate (F) at (-2,0,1);7
		\coordinate (G) at (-2,0,-1);8
		\coordinate (bA) at (1,2,0);9
		\coordinate (bE) at (-1,2,0);10
		\coordinate (C) at (0,-1,-2);11
		\coordinate (bI) at (0,1,-2);12
		
		\draw (A)--(B)--(C)--(A);
		\draw (B)--(E)--(C);
		\draw (A)--(D)--(B);
		\draw (B)--(F)--(E);
		\draw (E)--(G)--(C);
		\draw (A)--(H)--(C);
		\draw (D)--(I)--(F);
		\draw[dashed] (F)--(D);
		\draw (A)--(bA)--(D);
		\draw (H)--(bA);
		\draw[dashed] (E)--(bE);
		\draw[dashed] (I)--(bA);
		\draw[dashed] (I)--(bE);
		\draw[dashed] (I)--(bI);
		\draw[dashed] (bI)--(H);
		\draw[dashed] (bI)--(G);
		\draw[dashed] (bI)--(bA);
		\draw[dashed] (bI)--(bE);
		\draw[dashed] (H)--(G);
		\draw[dashed] (G)--(bE);
		\draw[dashed] (F)--(bE);
		\end{tikzpicture}
		\caption{Jessen's orthogonal icosahedron}
		\label{fig:jessen}
	\end{figure}
\end{example}

Note that the property of being isothermic is M\"{o}bius invariant.

\begin{corollary}
	The infinitesimal rigidity of a non-degenerate simply connected triangulated surface inscribed in a sphere is M\"{o}bius invariant.
\end{corollary}

We can regard a M\"{o}bius transformation of a triangulated surface inscribed in a sphere as being induced from a projective transformation of the ambient space. Then the above corollary is simply a special case of the projective invariance of infinitesimal rigidity \cite{Izmestiev2008}.

\section{Discrete minimal surfaces}\label{sec:minimalsurf}

In the smooth theory a minimal surface is the Christoffel dual of its Gau{\ss} map. We give a definition of discrete minimal surfaces that is inspired by this fact.

\begin{definition}
	Given a non-degenerate realization $f:V \to \field{R}^3$ of a triangulated surface, a non-constant map $f^*: F \to \field{R}^3$ is called a \emph{Christoffel dual} of $f$ if
	\begin{align*}
		d\!f(e_{i\!j}) \times d\!f^*(e^*_{i\!j}) &= 0  \quad \forall \{i\!j\} \in E_{int},\\
		\sum_{j} \langle d\!f(e_{i\!j}), d\!f^*(e^*_{i\!j}) \rangle &=0 \quad \forall i \in V_{int}.
	\end{align*}
\end{definition}

Since we know from the previous section that triangulated disks inscribed in a sphere with more than $3$ boundary vertices are isothermic, the following definition is natural.
\begin{definition}
	Given a non-degenerate realization $f:V \to \field{S}^2$ of a triangulated surface, its Christoffel dual $f^*:F \to \field{R}^3$ is called a \emph{discrete minimal surface} with Gauss map $f$.
\end{definition}

Equivalently, a discrete minimal surface is a \emph{reciprocal-parallel mesh} of an inscribed triangulated surface. The combinatorics of a discrete minimal surface is that of the dual cell complex and each dual edge is parallel to the corresponding primal edge. Figure \ref{fig:enneper} shows a discrete minimal surface together with its Gauss map.

Starting with a planar triangulated disk $f$ together with a discrete harmonic function $u$ we obtain a dual 1-form $\tau$ satisfying Definition \ref{def:isoth}. Using Lemma \ref{lem:trans} we can apply a stereographic projection $\Phi$ to $f$ in order to obtain a dual 1-form $\tau_{\Phi}$ for $\Phi\circ f$. Then integrating the dual 1-form $\tau_{\Phi}$ on the dual mesh yields a discrete minimal surface. 
\begin{corollary}\label{thm:discreteminhar}
	Every simply connected discrete minimal surface is given by a discrete harmonic function on a planar triangular mesh.
\end{corollary}
\begin{figure}[ht]
	\centering
	\includegraphics[width=\textwidth]{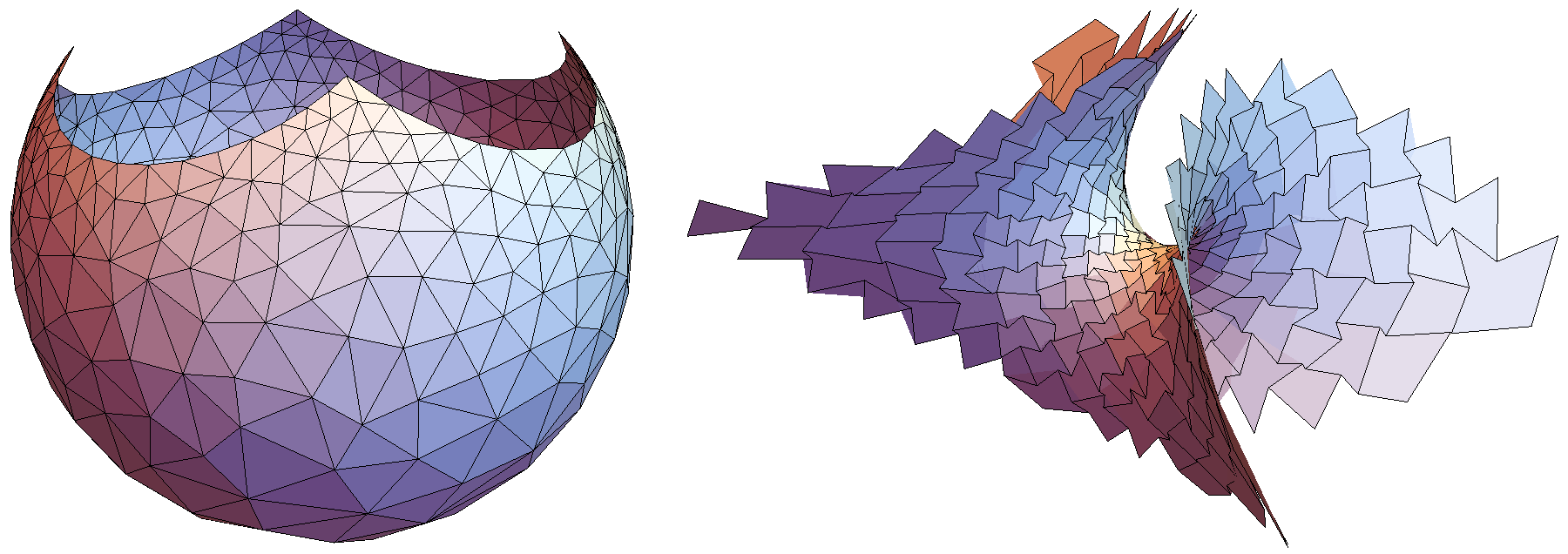}
	\caption{Right: A discrete Enneper surface corresponding to the discrete harmonic function $u|_{\partial M}=xy$. The edges are parallel to the primal edges of the Gau{\ss} image shown on the left.}
	\label{fig:enneper}
\end{figure}

As an example we take a triangulated square $f$ in the $(x,y)$-plane (Figure \ref{fig:enneper} left shows the stereographic projection of $f$). Notice that on a Delaunay triangulated disk $D$ in the plane, a discrete harmonic function $u$ is uniquely determined by its boundary values. The choice $u|_{\partial D}=xy$ leads to a discrete Enneper surface (Figure \ref{fig:enneper} right).
\begin{figure}[ht]
	\includegraphics[width=0.9\textwidth]{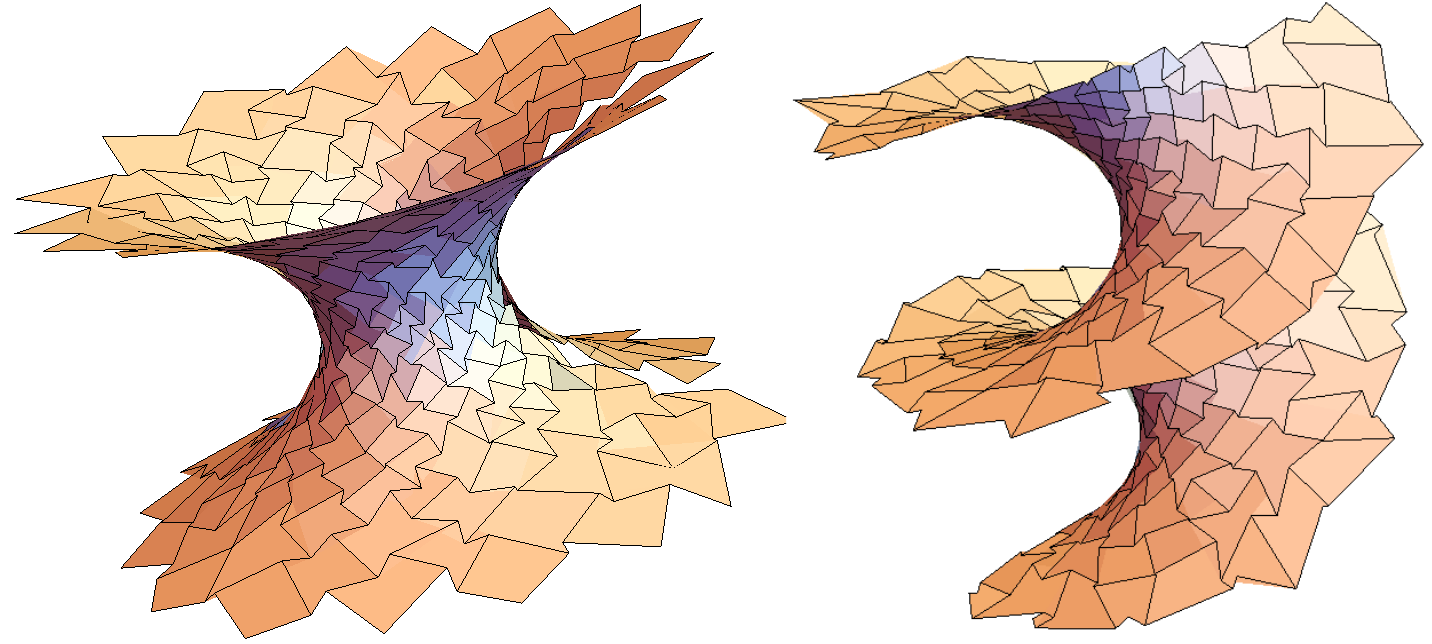}
	\caption{A discrete catenoid and a discrete helicoid.}
	\label{fig:catenoid}
\end{figure}
As a second example, we consider a triangulated annulus centered at the origin with a cut along the positive x-axis. We solve for the discrete harmonic function $u$ with boundary values either given by $u|_{\partial M} = \log |z|$ or by $u|_{\partial M} = \arg z$. We obtain a discrete helicoid and a discrete catenoid (Figure \ref{fig:catenoid}) respectively.

Theorem \ref{thm:discreteminhar} indicates that the Weierstrass data needed to construct a discrete minimal surface consist of a discrete harmonic function on a planar triangular mesh. Its smooth counterpart is the Weierstrass representation for minimal surfaces, which states that each minimal surface locally is given by a pair of holomorphic functions. The choice of such a pair of holomorphic functions is equivalent to prescribing the Gauss map and the Hopf differential of a minimal surface. A similar formula for discrete minimal surfaces is elaborated in \cite{Lam2015}, where a notion of holomorphic quadratic differentials is introduced to link discrete harmonic functions and discrete minimal surfaces by considering infinitesimal conformal deformations. Furthermore, it turns out that our definition of discrete minimal surfaces bridges the gap between earlier notions of discrete minimal surfaces \cite{Lam2015b}.

\section{Smooth analogues} \label{sec:smooth}

The main goal of this section is to prove the smooth analogue of Theorem \ref{thm:sphiso}:
\begin{theorem}\label{thm:sspiso}
	For every infinitesimal isometric deformation of an immersion $f:M \to S^2$, the mean curvature is preserved. 
\end{theorem}

Beyond this we also use the opportunity to review some known results on smooth isothermic surfaces that directly correspond to our discrete results. We rely on the treatment of smooth isothermic surfaces by means of quaternionic analysis as developed in  \cite{Kamberov2002,Kamberov1998,Richter1997}.

\begin{definition}
	Given two immersions $f$ and $\tilde{f}$ of a surface M in $\field{R}^3 \cong \Im\mathds{H}$, $\tilde{f}$ is a spin transformation of $f$ if there exists a quaternion-valued function $\lambda : M \to \mathds{H}\backslash \{0\} $ such that
	\begin{equation*}
		d\!\tilde{f}=\bar{\lambda}d\!f\,\lambda.
	\end{equation*}
	In this case, the normals $N,\tilde{N}$ of $f,\tilde{f}$ are related by $\tilde{N}= \lambda\!^{-1} N \lambda$.
\end{definition}
\begin{theorem}
	Given an immersion $f:M \to \field{R}^3\cong \Im\mathds{H}$ and a function $\lambda : M \to \mathds{H}\backslash \{0\} $. Then $\bar{\lambda}d\!f\,\lambda$ is a closed 1-form if and only if there exists $\rho:M \to \field{R}$ such that 
	\[
	d\!f \wedge d\lambda = -\rho \lambda |d\!f|^2.
	\]
\end{theorem}

Suppose $f:M \to \field{R}^3$ is a smoothly immersed surface and $X,JX \in T_p M$ form an orthonormal basis in principal directions at $p \in M$. The corresponding principal curvatures $\kappa_1,\kappa_2 \in \field{R}$ is given by
\[
dN(X_i) = \kappa_i d\!f(X_i). 
\]
The mean curvature $H_p$ at $p$ then satisfies
\[
d\!f(X) \times dN(JX) - d\!f(JX) \times dN(X)  = (\kappa_1 + \kappa_2) N = 2 H N. 
\]
Writing the above formula into quaternions yields
\[
d\!f \wedge dN = 2HN |d\!f|^2.
\]
If $\tilde{f}$ is a spin transformation of $f$, then the change of the metric and the mean curvature can be expressed as follows.
\begin{theorem}
	If $\tilde{f}$ is a spin transformation of $f$ given by $\lambda$ then the followings hold:
	\begin{enumerate}
		\item $\tilde{f}$ and $f$ are conformally equivalent since
		$|d\! \tilde{f}|^2=|\lambda|^4 |d\!f|^2$.
		\item We have $d(\bar{\lambda}d\!f\,\lambda) =0$ and hence $\exists \rho:M \to \field{R}$ such that $d\!f \wedge d\lambda = -\rho \lambda |d\!f|^2$.
		\item $\tilde{H} |d\!\tilde{f}|^2 = (H + \rho) |\lambda|^2 |d\!f|^2$.
	\end{enumerate}
\end{theorem}
Now we use the fact that every infinitesimal conformal deformation can be expressed as an infinitesimal spin transformation.
\begin{theorem}\label{thm:infspin}
	Suppose $f: M \to \field{R}^3=\Im \field{H}$ is a simply connected smooth surface and $\dot{\lambda}:M \to \field{H}$ is a function. Then there exists an infinitesimal conformal deformation $\dot{f}: M \to \field{R}^3=\Im \field{H}$ given by
	\[
	d\!\dot{f} = 2\Im(d\!f\,{\dot{\lambda}})
	\]
	if and only if there exists $\dot{\rho}:M \to \field{R}$ such that
	\[
	-d\!f \wedge d\dot{\lambda} = \dot{\rho}|d\!f|^2.
	\]
	In particular 
	\begin{align*}
		(|d\! f|^2)^{\LargerCdot} &=  4 \Re(\dot{\lambda}) |d\!f|^2, \\
		(H |d\!f|^2)^{\LargerCdot} &= (\dot{\rho}+ 2\Re(\dot{\lambda})H) |d\!f|^2.
	\end{align*}
	%
\end{theorem}

\begin{definition}\label{def:smoothiso}
	A smooth surface $f: M \to \field{R}^3\cong \Im \field{H}$ is called isothermic if locally there exists a non-trivial $\Im \field{H}$-valued closed 1-form $\tau$ such that
	\[
	d\!f \wedge \tau =0.
	\] 
\end{definition}

\begin{theorem}
	A smooth surface $f: M \to \field{R}^3 \cong \Im \field{H}$ is isothermic if and only if locally there exists a non-trivial infinitesimal isometric deformation such that the mean curvature is unchanged.
\end{theorem}
\begin{proof}
	Let $f$ be isothermic and $\tau$ be a 1-form satisfying Definition \ref{def:smoothiso}. The closedness of $\tau$ implies that on any simply connected open set $U\subset M$ there exists $\dot{\lambda}: U \to \Im \field{H}$ such that
	\[
	\tau = d\dot{\lambda}.
	\] 
	Since $d\!f \wedge d\dot{\lambda} =0$, Theorem \ref{thm:infspin} implies that the 1-form $\Im(d\!f \,\dot{\lambda})$ is closed and hence there exists an infinitesimal deformation $\dot{f}:U \to \Im \field{H}$ satisfying
	\begin{gather*}
		d\!\dot{f} = 2\Im(d\!f \, \dot{\lambda}).
	\end{gather*}
	Because $\dot{\lambda}$ is purely imaginary and $d\!f \wedge d\dot{\lambda} =0$, from Theorem \ref{thm:infspin} we have
	\begin{align*}
		(|d\! f|^2)^{\LargerCdot} &= 0, \\
		(H |d\!f|^2)^{\LargerCdot} &= 0
	\end{align*}
	and hence $\dot{f}$ is an infinitesimal isometric deformation preserving the mean curvature. The converse is proved similarly. 
\end{proof}

The smooth analog of Theorem \ref{thm:mob} is a classical result of isothermic surfaces \cite{Hertrich-Jeromin2003}.
\begin{theorem} \label{thm:smoothmob}
	The class of isothermic surfaces is M\"{o}bius invariant.
\end{theorem}

We need the global existence of $\tau$ in order to relate it to the space of immersions.

\begin{definition}\label{def:stonglysmoothiso}
	A smooth surface $f: M \to \field{R}^3 \cong \Im \field{H}$ is called \emph{strongly isothermic} if there exists a non-trivial $\Im \field{H}$-valued closed 1-form $\tau$ on $M$ such that
	\[
	d\!f \wedge \tau =0.
	\]
	In addition, if $\tau$ is exact and $\tau = d\!f^*$ for some $f^*: M \to \field{R}^3 \cong \Im \field{H}$, then $f^*$ is called a Christoffel dual of $f$. 
\end{definition}

The smooth analog of Corollary \ref{cor:conflength} in Section \ref{sec:infcon} is known:
	\begin{theorem}[\cite{Bohle2008}]
		Strongly isothermic immersions of a closed surface are the points in the space of immersions where the map from the space of immersions to Teich\-m\"{u}l\-ler space (which assigns to each immersion the conformal class of its induced metric) fails to be a submersion.
	\end{theorem}

Furthermore, Theorem \ref{thm:intangles} is analogous to the infinitesimal version of the following:
\begin{theorem}[\cite{Burstall2002}]
	Suppose $f,\tilde{f}: M \to S^3$ are two conformal immersions that share the same Hopf differential but do not differ by a M\"{o}bius transformation. Then $f$ and $\tilde{f}$ are isothermic surfaces.
\end{theorem}

Finally, we establish a smooth counterpart of Theorem \ref{thm:sphiso} in Section \ref{sec:inscribedmesh}. We first need a lemma.

\begin{lemma} \label{thm:vector}
	Let $M$ be a surface with a Riemannian metric and $f:M \to \field{R}^3$ be an isometric immersion. Let $\lambda = g + d\!f(Y) + hN$ be an $\field{H}$-valued function on $M$ where $g$,$d\!f(Y)$ and $hN$ are its scalar, tangential and normal components. Then we can express $d\!f\wedge d\lambda$ in terms of standard operators from the vector calculus on $M$ as
	\[
	-d\!f \wedge  d\lambda = \big[ -\curl Y  + d\!f (J\grad g - \A Y +\grad h) - \big((\div Y)+ 2hH\big)N \big] |d\!f|^2.
	\]
\end{lemma}

\begin{proof}
	In the following we assume that $X\in T_pM$ is an unit tangent vector. We first consider the scalar component.
	\begin{align*}
		-d\!f \wedge dg (X,JX) &= -d\!f(X)dg(JX) + d\!f(JX)dg(X) \\
		&= N d\!f(dg(X) X + dg(JX) JX) \\
		&= N d\!f (\grad g).
	\end{align*}	
	Then we consider the normal component.
	\begin{align*}
		-d\!f \wedge d(hN) (X,JX) &= \big( (-d\!f \wedge dh)N - h d\!f \wedge dN\big) (X,JX) \\
		&= N d\!f(\grad h) N - h\big(d\!f(X) dN(JX) - d\!f(JX)dN(X)\big) \\
		&= d\!f(\grad h) - 2h H N.
	\end{align*}
	Finally we look at the tangential component. Notice that for an immersed surface in Euclidean space the induced Levi-Civita connection is given as follows: for any tangent vector field $Y$ and tangent vector $Z$,
	\begin{align*}
		d\!f(\nabla_Z Y)&=d(d\!f(Y))(Z) - \langle d(d\!f(Y))(Z),N \rangle N \\
		&=d(d\!f(Y))(Z) + \langle d\!f(Y),d\!f(AZ) \rangle N \\
		&=d(d\!f(Y))(Z) + \langle Y, AZ \rangle N
	\end{align*}
	where $A$ is the shape operator of the immersion $f$. We recall the definition of curl and divergent operator of a tangent vector field $Y$:
	\begin{align*}
		\div(Y) :&= \langle X , \nabla_X Y \rangle + \langle X , \nabla_X Y \rangle, \\
		\curl(Y) :&= \langle JX , \nabla_X Y \rangle - \langle X , \nabla_{JX} Y \rangle \\
		&= -\langle X , \nabla_X JY \rangle - \langle JX , \nabla_{JX} JY \rangle \\
		&= -\div(JY)	.
	\end{align*}
	Collecting the above information we now obtain
	\begin{align*}
		&-d\!f \wedge d(d\!f(Y)) (X,JX) \\ =& -d\!f(X)\big(d\!f(\nabla_{JX} Y) - \langle Y,AJX \rangle N \big) 
		+ d\!f(JX)\big(d\!f(\nabla_{X} Y) - \langle Y,AX \rangle N \big) \\
		=& \langle X,\nabla_{JX} Y \rangle - \langle JX, \nabla_X Y \rangle - \langle JX, \nabla_{JX} Y \rangle N \\ &+ \langle  -X , \nabla_X Y \rangle N - \langle AY,JX \rangle d\!f(JX) - \langle AY,X \rangle d\!f(X) \\
		=& -\curl Y - (\div Y)N - d\!f(\A Y). 
	\end{align*} 
\end{proof}

\begin{proof}[Proof of Theorem \ref{thm:sspiso}]
	Suppose $f: M \to S^2$ is an immersion and $\dot{\lambda}: M \to \Im(\field{H})$ induces an infinitesimal isometric deformation of $f$. Then writing $\dot{\lambda} = d\!f(Y) + hN$, we have
	\[
	(H|d\!f|^2)^{\LargerCdot} = -d\!f \wedge d\dot{\lambda} = -\curl Y + d\!f(-Y + \grad h)- ((\div Y + 2h))N.
	\]
	Comparing the imaginary part yields
	\[
	Y = \grad h
	\]
	and thus
	\[
	(H|d\!f|^2)^{\LargerCdot} = - \curl(Y) = - \curl(\grad h) =0.
	\] 
	Since $(|d\!f|^2)^{\LargerCdot}=0$, we have $\dot{H}=0$. 
\end{proof}

\bibliographystyle{spmpsci}      

\end{document}